\documentclass[12pt,reqno]{amsart}

\usepackage{fullpage}
\usepackage{amsfonts}
\usepackage{amssymb}
\usepackage{times}
\usepackage{graphicx}
\usepackage{tgpagella}
\usepackage{euler}
\usepackage[T1]{fontenc}
\usepackage{amssymb}
\usepackage{amsmath,amsthm}
\usepackage[colorlinks=true,
            linkcolor=red,
            urlcolor=blue,
            citecolor=blue]{hyperref}
\usepackage{amsthm}
\usepackage[english]{babel}
\usepackage{mathrsfs}
\usepackage{eucal}
\usepackage{calligra}
\usepackage{calrsfs}
\usepackage{enumitem}
\usepackage{mathpazo}

\numberwithin{equation}{section}

\newtheorem{main}{Theorem}

\newtheorem{mcor}[main]{Corollary}

\newtheorem{theorem}{Theorem}[section]
\newtheorem{lem}[theorem]{Lemma}
\newtheorem{prop}[theorem]{Proposition}

\newtheorem{cor}[theorem]{Corollary}

\theoremstyle{definition}

\newtheorem{notation}[theorem]{Notation}
\newtheorem{Remarks}[theorem]{Remarks}

\newtheorem{claim}[theorem]{Claim}
\newtheorem*{claim*}{Claim}
{ \theoremstyle{remark} }

\DeclareMathAlphabet{\pazocal}{OMS}{zplm}{m}{n}

\def\ca{\curvearrowright}
\def\ra{\rightarrow}

\def\la{\lambda}
\def\La{\Lambda}
\def\Om{\Omega}

\def\Sg{\Sigma}

\def\g{\gamma}

\def\G{\Gamma}

\newcommand{\Z}{\mathbb{Z}}

\newcommand{\C}{\mathbb{C}}

\theoremstyle{remark}

\newcommand{\set}[1]{ \{ #1 \} }
\newcommand{\norm}[1]{\| #1 \|}

\newcommand{\generator}[1]{\langle #1 \rangle}

\begin{document}

\title[Wreath product von Neumann algebras]{Some rigidity results for II$_1$ factors arising from wreath products of property (T) groups}
\author[I. Chifan]{Ionut Chifan}
\address{Department of Mathematics, The University of Iowa, 14 MacLean Hall, IA  
52242, USA}
\email{ionut-chifan@uiowa.edu}
\thanks{I.C.\ was supported by NSF Grant DMS \#1600688}
\author[B. Udrea]{Bogdan Teodor Udrea}
\address{Department of Mathematics, The University of Iowa, 14 MacLean Hall, IA  
52242, USA}
\email{bogdan-udrea@uiowa.edu}

\maketitle

\begin{abstract}We show that any infinite collection $(\G_n)_{n\in \mathbb N}$ of icc, hyperbolic, property (T) groups satisfies the following von Neumann algebraic \emph{infinite product rigidity} phenomenon. If $\La$ is an arbitrary group such that $L(\oplus_{n\in \mathbb N} \G_n)\cong L(\La)$ then there exists an infinite direct sum decomposition $\La=(\oplus_{n \in \mathbb N} \La_n )\oplus A$ with $A$ icc amenable such that, for all $n\in \mathbb N$, up to amplifications, we have $L(\G_n) \cong L(\La_n)$  and $L(\oplus_{k\geq n} \G_k )\cong L((\oplus_{k\geq n} \La_k) \oplus A)$. The result is sharp and complements the previous finite product rigidity property found in \cite{CdSS16}. Using this we provide an uncountable family of restricted wreath products $\G\cong\Sigma\wr \Delta$ of icc, property (T) groups $\Sigma$, $\Delta$ whose wreath product structure is recognizable, up to a normal amenable subgroup, from their von Neumann algebras $L(\G)$.  Along the way we highlight several applications of these results to the study of rigidity in the $\mathbb C^*$-algebra setting.
\end{abstract}






\maketitle


\section{Introduction}

For a countable infinite group $\Gamma$ we denote by $\ell^2\Gamma$ the Hilbert space of all square summable complex functions on $\Gamma$. Each element $\g\in \Gamma$ gives rise to a unitary operator $u_\g :\ell^2\G\rightarrow \ell^2\G$ by group translation $u_\g (\xi)(\la)= \xi(\g^{-1}\la)$, where $\la\in \G$ and  $\xi\in \ell^2\G$. The bicommutant $\{ u_\g | \g\in \G\}''$ inside the algebra of all bounded linear operators $B(\ell^2 \G)$, is denoted by $L(\G)$ and it is called the \emph{group von Neumann algebra of $\G$}. The algebra $L(\G)$ is a II$_1$ factor (has trivial center) precisely when all nontrivial conjugacy classes of $\G$ are infinite (icc), this being the most interesting for study \cite{MvN43}. 
\vskip 0.02in
Ever since their introduction, the classification of these factors is a core direction of research driven by the following fundamental question: \emph{What aspects of the group $\Gamma$ are remembered by $L(\Gamma)$?} This emerged as an interesting yet intriguing theme since these algebras tend to have little memory of the initial group. This is best illustrated by Connes' celebrated result asserting that all amenable icc groups give isomorphic factors, \cite{Co76}. Hence very different groups like the group of all finite permutations of the positive integers, the lamplighter group, or the wreath product of the integers with itself give rise to isomorphic factors. Consequently, the von Neumann algebric structure has no memory of the typical discrete algebraic group invariants like torsion, rank, or generators and relations. In this case the only information the von Neumann algebra retains is the amenability---an approximation property---of the group.
\vskip 0.02in
In the non-amenable case the situation is radically different and an unprecedented progress has been achieved through the emergence of Popa's deformation/rigidity theory \cite{Po06}. Using this completely new conceptual framework it was shown that various properties of groups, such as their representation theory or their approximations, can be completely recovered from their von Neumann algebras. As a result, for large classes of group factors, many remarkable structural properties such as primeness, (strong) solidity, classification of normalizers of algebras, etc could be successfully established \cite{Po03,IPP05,Po06,Po08,OP07,OP08,CH08,CI08,Pe09,PV09,FV10,Io10,IPV10,HPV10,CP10,Si10,Va10,CS11,CSU11,Io11,PV11,HV12,PV12,Io12,Bo12,BHR12,Is12,BV13,Va13,Is14, CIK13,VV14,BC14,CKP14,CdSS16,DHI16,CdSS17}. For additional information we refer the reader to the following survey papers \cite{Po06,V10,Io12,Io17icm}.

One of the most impressive milestone in this study is Ioana-Popa-Vaes' discovery of the first examples of groups that can be completely recovered from their von Neumann algebras (\emph{$W^*$-superrigid}\footnote{$\G$ is $W^*$-superrigid if whenever $\La$ is an arbitrary group so that  $L(\G) =L(\La) $ then it follows  that  $\La=\G$.} groups), \cite{IPV10}. See also the subsequent result \cite{BV13} and the more recent work \cite{CI17}. These results pushed the classification problem of group factors to new boundaries and exciting possibilities. In this direction an interesting and wide open theme is to identify a comprehensive list of canonical constructions in group theory (direct sum, free product, HNN-extension, wreath product, etc) that are recoverable from their von Neumann algebras.

\subsection{Statements of main results} Over a decade ago Ozawa and Popa discovered the first unique prime factorization results for tensor product of II$_1$ factors, \cite{OP03}. This work has had deep consequences to the classification of II$_1$ factors and has generated significant subsequent developments.  
Some of  Ozawa-Popa's results have been strengthened considerably in \cite{CdSS16} where was unveiled a large class of product groups $\G_1\times \G_2$ whose product structure is a feature completely recognizable at the level of their von Neumann algebras $L(\G_1\times \G_2)$. Precisely, whenever $\G_1,\G_2$ are hyperbolic icc groups (e.g.\ non-abelian free groups) and $\La$ is an \emph{arbitrary} group such that $L(\G_1\times \G_2)=L(\La)$ then $\La$ admits a nontrivial product decomposition  $\La=\La_1\times \La_2$ and there exists a scalar $t>0$ such that, up to unitary conjugacy, we have $L(\G_1)=L(\La_1)^t$ and $L(\G_2)=L(\La_2)^{1/t}$. The result still holds if one assumes, more generally, that $\G_1,\G_2$ are just icc biexact groups, \cite{Oz03}. 

Isono studied unique prime factorization aspects for infinite tensor product of factors and several interesting results have emerged in \cite{Is16}. Motivated in part by these results it is natural to investigate whether ``product rigidity'' properties, similar with ones in \cite{CdSS16},  would hold in the context of infinite direct sums groups. Specifically, if one considers $\G=\oplus_{n\in \mathbb N} \G_n$ with $\G_n$'s icc non-amenable groups, it would be interesting to understand how much of the infinite direct sum structure of $\G$ is retained by its von Neumann algebra $L(\G)$.  Right away one may notice a sharp contrast point with the aforementioned finite product situation. Since $L(\G)$ canonically decomposes as an infinite tensor product $L(\G)= \bar \otimes_{n\in \mathbb N} L(\G_n)$ it follows that $L(\G)$ is a McDuff factor and hence  $L(\G)=L(\G)\bar\otimes \mathcal R$, where $\mathcal R$ is the hyperfinite factor; consequently, we have that $L(\oplus_{n\in \mathbb N} \G_n)= L((\oplus_{n\in\mathbb N}\G_n) \oplus A)$, for any icc amenable group $A$. This observation shows that, in the best case scenario, $L(\G)$ could remember the direct sum feature of the underlying group \emph{only up to an amenable subgroup} which typically lies in the tail of the infinite tensor product. It is therefore natural to investigate under which circumstances it is possible to completely reconstruct the infinite direct sum feature only up to this obstruction. Building upon previous techniques from \cite{IPV10,Io11,CdSS16,DHI16,CdSS17} and using the classification of normalizers from \cite{PV12} we found infinitely many classes of $\G_n$'s for which this problem has a  positive answer.     

\begin{main}\label{main1} Let $(\G_n)_{n\in \mathbb N} $ an infinite collection of property (T), biexact, weakly amenable, icc groups. Assume that $\Lambda$ is an \emph{arbitrary} group satisfying $L(\oplus_{n\in \mathbb N} \G_n)=L(\Lambda)$. Then $\La$ admits an infinite direct sum decomposition $\La=(\oplus_{n\in\mathbb N} \La_n) \oplus A$, where $\La_n$ is icc, weakly amenable, property (T) group for all $n$ and $A$ is a icc amenable group. Moreover, for each $k\in \mathbb N$ there exist scalars $t_1,t_2,...,t_{k+1}>0$ satisfying $t_1 t_2 ... t_{k+1}=1$ and a unitary $u\in L(\La)$ so that 
\begin{equation*}\begin{split}&uL(\G_n)^{t_n}u^*=L(\La_n)\quad \text{ for all }k\geq n\geq 1 \text{; and }\\& uL(\oplus_{n\geq k+1} \G_n)^{t_{k+1}}u^*=L((\oplus_{n\geq k+1} \La_n) \oplus A).\end{split}
\end{equation*}  
\end{main} 

The result applies to several concrete classes of groups a such as: \begin{enumerate}\item the uniform lattices $\G_n$ in $Sp(k_n,1)$ with $k_n\geq 2$ or any icc groups in their measure equivalence class; and \item Gromov's random groups with density satisfying $3^{-1}<d<2^{-1}$.\end{enumerate}   

While the conclusion of the previous theorem shows a strong identification of the von Neumann algebras of the group factors $\G_n$'s, in general one cannot recover these groups. To see this note that Voiculescu's compression formula for free group factors gives  $L(\mathbb F_2)= L(\mathbb F_5)\otimes M_2(\mathbb C)$. This implies that $L(\oplus_{n\in \mathbb N}  \mathbb F_2)=\bar\otimes_{n\in \mathbb N} L(\mathbb F_2) = \bar\otimes_{n\in \mathbb N} (L(\mathbb F_5)\otimes M_2(\mathbb C)) =(\bar\otimes_{n\in \mathbb N} (L(\mathbb F_5))\bar \otimes \mathcal R = L((\oplus_{n\in \mathbb N} \mathbb F_5) \oplus A)$, for every icc amenable group $A$.

 Theorem \ref{main1} can be successfully used to shed light towards rigidity aspects in the $ C^*$-algebraic setting. Precisely, when it is combined with \cite[Theorem 1.3]{BKKO14} one gets the following version of infinite product rigidity for reduced group $C^*$-algebras.

\begin{mcor}\label{maincor1} Let $(\G_n)_{n\in \mathbb N} $ an infinite collection of property (T), biexact, weakly amenable, icc groups. Assume that $\Lambda$ is an \emph{arbitrary} group satisfying $\mathbb C^*_r(\oplus_{n\in \mathbb N} \G_n)=\mathbb C^*_r(\Lambda)$. Then $\La$ admits an infinite direct sum decomposition $\La=\oplus_{n\in\mathbb N} \La_n$, where the $\La_n$'s are icc, weakly amenable, property (T) groups.
\end{mcor} 

When compared side by side, Theorem \ref{main1} and Corollary \ref{maincor1} highlight again the fundamental difference between $\mathbb C^*$-algebras  and von Neumann algebras; the absence of the infinite amenable direct summand of $\La$ in the conclusion of Corollary \ref{maincor1} exemplifies once more the fact that the WOT closure is considerably larger than the normic closure, thus triggering significant loss of algebraic information in the von Neumann algebraic setting.
\vskip 0.04in 

Restricted wreath product groups manifest a remarkable rigid behavior in the von Neumann algebraic setting. In fact a large portion  of the groups/actions known to be  reconstructible of from their von Neumann algebras arise from constructions involving wreath product groups or Bernoulli shifts \cite{Po03,Po05,PV09,Io10,IPV10,PV11,BV13}. A common feature of these examples is that the core or the wreath product groups involved are amenable and in many cases even abelian. For example, with the exception of \cite{CI17}, all known examples of $W^*$-superrigid\footnote{a group $K$ is \emph{$W^*$-superrigid} if whenever $T$ is an \emph{arbitrary} group such that $L(K)=L(T)$ then $K=T$.} groups are of the form $H\wr_I \G$ where $H$ is finite \cite{IPV10,BV13,B14}. However significantly fewer rigidity results are known in general for wreath product factors $L(H\wr_I \G )$ when $H$ is nonamenable. In this direction we mention in passing Ioana's strong rigidity results which asserts that $L(\mathbb F_n\wr A )\neq L(\mathbb F_m \wr B)$ whenever $n,m\geq 2$ and $A$, $B$ are nonisomorphic icc amenable groups, \cite{Io06}. Theorem \ref{main1} can be successfully used to provide new insight towards this problem as well. For instance, using it in combination with various technical outgrowths of previous methods from \cite{Po03,Io06,IPV10,CdSS16,CI17} we obtain the following \emph{wreath product rigidity} result \emph{up to an amenable subgroup} for group factors.

\begin{main}\label{main4} Let $H$ be icc, weakly amenable, biexact property (T) group. Let $\G=\G_1\times \G_2$, where $\G_i$ are icc, biexact, property (T) group.  
Let $H \wr \G$ be the corresponding (plain) wreath product. Let $\La$ be an \emph{arbitrary} group and let $\theta: L(H \wr \G)\ra L(\La)$ be a $\ast$-isomorphism. Then one can find non-amenable icc groups  $\Sg,\Psi$, an amenable icc group $A$, and an action $\Psi \ca^\alpha A$ such that we can decompose $\La$ as semidirect product $\La = (\Sg^{(\Psi)} \oplus A)\rtimes_{\beta\oplus \alpha} \Psi$, where $\Psi\ca^{\beta} \Sg^{(\Psi)}$ is the Bernoulli shift. In addition, there exist a group isomorphism $\delta:\G \ra \Psi$, a character $\eta:\G\ra \mathbb T$, a $\ast$-isomorphism $\theta: L(H^{(\G)})\ra L(\Sg^{(\Psi)}\oplus A)$ and a unitary $u\in L(\La)$ so that for all $x\in L(H^{(\G)})$, $\g\in \G$ we have \begin{equation*}
\theta(x u_\g) =\eta(\g)u \theta(x) v_{\delta(\g)}u^*.
\end{equation*}
Here $\{u_\g \,|\, \g\in \G\}$ and $\{v_\la \,|\, \la\in \Psi\}$ are the canonical group unitaries of $L(\G)$ and $L(\Psi)$, respectively. 
\end{main}
We notice the theorem still holds for slightly more general situations, e.g.\ generalized wreath products with finite or even possible just amenable stabilizers  (see Theorem \ref{wprig2}). However at this time we do not know the full extent of these cases as it seems to rely on heavy constructions on group theory (see Remarks \ref{possibleex}).

As before, Theorem \ref{main4} in combination with \cite[Theorem 1.3]{BKKO14} lead to the following stronger version of wreath product rigidity in the context of reduced group $\mathbb C^*$-algebras.

\begin{mcor}\label{maincor2} Let $H$ be icc, weakly amenable, biexact property (T) group. Let $\G=\G_1\times \G_2$, where $\G_i$ are icc, biexact, property (T) group.  
Let $H \wr \G$ be the corresponding wreath product. Let $\La$ be an \emph{arbitrary} group so that $\mathbb C^*_r( H \wr \G)= \mathbb C^*_r(\La)$. Then $\La =\Sg\wr \G$, where $\Sg$ is an icc, weakly amenable, property (T) group. 
\end{mcor}

In connection with Theorem \ref{main4} it is natural to investigate whether similar statements hold if one relaxes the biexactness assumptions on $H$ or the product assumption $\G$. In this situation, building upon the previous techniques from \cite{IPV10,KV15} one can show the following strong rigidity statement holds just for property (T) groups $H$ and $\G$.

\begin{main}\label{main3} Let $H, \G$  be icc, torsion free groups. Also assume $H$ has property (T) and $\G$ admits an infinite, almost normal subgroup with relative property (T). Let $\Gamma \curvearrowright I$ be a transitive action on a countable set satisfying the following conditions:
\begin{enumerate}
\item [a)] The stabilizer ${\rm Stab}_\G(i)$ has infinite index in $\G$ for each $i \in I$;
\item [b)] There is $k\in \mathbb N$ such that for each $J\subseteq I$ satisfying  $|J|\geq k$ we have $|{\rm Stab}_\G(J)|<\infty$;
\item [c)] The orbit ${\rm Stab}_\G(i)\cdot j$ is infinite for all $i\neq j$. 
\end{enumerate}
Denote by $G=H \wr_I \G$ the corresponding generalized wreath product. Let $\La$ be any \emph{torsion free} group and let $\theta: L(G)\ra L(\La)$ be a $\ast$-isomorphism. Then $\La$ admits a generalized wreath product decomposition $\La = \Sg \wr_I \Psi$ satisfying all the properties enumerated in ${\rm a)-c)}$. In addition, there exist a group isomorphism $\delta:\G \ra \Psi$, a character $\eta:\G\ra \mathbb T$, a $\ast$-isomorphism $\theta: L(H)\ra L(\Sg)$ and a unitary $u\in L(\La)$ such that for every $x\in L(H^{(I)})$ and $\g\in \G$ we have \begin{equation*}
\theta(x u_\g) =\eta(\g)u \theta^{\otimes_I}(x) v_{\delta(\g)}u^*.
\end{equation*}
Here $\{u_\g | \g\in \G\}$ and $\{v_\la | \la\in \Psi\}$ are the canonical unitaries of $L(\G)$ and $L(\Psi)$, respectively.
\end{main}

The previous theorem applies to many natural families of generalized wreath groups $H\wr_I\G$, including: a) any icc torsion free prop (T ) group $H$ and any torsion free, hyperbolic, property (T) group $\G$ together with a maximal amenable subgroup $\Sg<\G$ and the action $\G\ca I=\G/\Om$ by translation on right cosets $\G/\Om$; b) (see \cite{IPV10}) Let $H$ be any icc property (T) group and $\G=\Z^2 \rtimes SL_2(\Z)$. Take a matrix $A$ in $SL_2(\Z)$ having the modulus of the eigenvalues larger than one, and the subgroup $S=\{B\in SL_2(\Z)\mid BAB^{-1}=\pm A\}$. The action $\G \curvearrowright I=\G/S$ is given by left translations.

\subsection{Acknowledgments} The first author is grateful to the American Institute of Mathematics for their kind hospitality during the workshop ``Classification of group von Neumann algebras'' where part of this work was carried out. The first author is also indebted to Professor Denis Osin for many discussions and suggestions regarding this project.

\section{Preliminaries}

\subsection{Notations} Given a von Neumann algebra 
$M$ we will denote by $\mathscr U(M)$ its unitary group and by $\mathscr P(M)$ the set of all its nonzero projections. All algebras inclusions $N\subseteq M$ are assumed unital unless otherwise specified. For any von Neumann subalgebras $P,Q\subseteq M$  we denote by $P\vee Q$ the von Neumann algebra they generate in $M$.

All von Neumann algebras $M$ considered in this article will be tracial, i.e., endowed with a unital, faithful, normal functional $\tau:M\ra \mathbb C$  satisfying $\tau(xy)=\tau(yx)$ for all $x,y\in M$. This induces a norm on $M$ by the formula $\|x\|_2=\tau(x^*x)^{1/2}$ for all $x\in M$. The $\|\cdot\|_2$-completion of $M$ will be denoted by $L^2(M)$.

For a countable group $\G$ we denote by $\{ u_\g | \g\in \G\} \in U(\ell^2\G)$ its left regular representation given by $u_\g(\delta_\la ) = \delta_{\g\la}$, where $\delta_\la:\G\ra \mathbb C$ is the Dirac mass at $\{\la\}$. The weak operatorial closure of the linear span of $\{ u_\g | \g\in \G\}$ in $B(\ell^2\G)$ is the so called group von Neumann algebra and will be denoted by $L(\G)$. $L(\G)$ is a II$_1$ factor precisely when $\G$ has infinite non-trivial conjugacy classes (icc).

Given a group $\G$ and a subset $F\subseteq \G$ we will be denoting by $\langle F\rangle$ the subgroup of $\G$ generated by $F$. Given a group action $\G\curvearrowright I$ on a countable set $I$, for any subset $\mathcal F\subset I$ we denote $Stab(\mathcal F)=\{\g\in\G\mid g\cdot i=i,\forall i\in \mathcal F\}$ and $Norm(\mathcal F)=\{\g\in\G\mid \g\cdot\mathcal F=\mathcal F\}$.

Given a subgroup  $\La \leqslant \G$ we will often consider the virtual centralizer of $\La$ in $\G$, i.e. $vC_\G(\La)=\set{\g\in \G : |\g^{\Lambda}|<\infty} $. Notice $vC_\G(\La)$ is a subgroup normalized by $\La$. When $\La=\G$, the virtual centralizer is denoted by $vZ(\G):=vC_\G(\G)$ and called the virtual center of $\G$; this is nothing else but the FC-radical of $\G$. Hence $\G$ is icc precisely when $vZ(\G)=1$.

\subsection{Popa's intertwining techniques}Over a decade ago, Popa introduced  in \cite [Theorem 2.1 and Corollary 2.3]{Po03} a powerful analytic criterion for identifying intertwiners between arbitrary subalgebras of tracial von Neumann algebras. This is now termed \emph{Popa's intertwining-by-bimodules technique}.

\begin {theorem}\cite{Po03} \label{corner} Let $(M,\tau)$ be a separable tracial von Neumann algebra and let $P, Q\subseteq M$ be (not necessarily unital) von Neumann subalgebras. 
Then the following are equivalent:
\begin{enumerate}
\item There exist $ p\in  \mathcal P(P), q\in  \mathcal P(Q)$, a $\ast$-homomorphism $\theta:p P p\rightarrow qQ q$  and a partial isometry $0\neq v\in q M p$ such that $\theta(x)v=vx$, for all $x\in p P p$.
\item For any group $\mathcal G\subset \mathscr U(P)$ such that $\mathcal G''= P$ there is no sequence $(u_n)_n\subset \mathcal G$ satisfying $\|E_{ Q}(xu_ny)\|_2\rightarrow 0$, for all $x,y\in  M$.
\end{enumerate}
\end{theorem} 
\vskip 0.02in
\noindent If one of the two equivalent conditions from Theorem \ref{corner} holds then we say that \emph{ a corner of $P$ embeds into $Q$ inside $M$}, and write $P\prec_{M}Q$. If we moreover have that $P p'\prec_{M}Q$, for any projection  $0\neq p'\in P'\cap 1_P M 1_P$ (equivalently, for any projection $0\neq p'\in\mathcal Z(P'\cap 1_P  M 1_P)$), then we write $P\prec_{M}^{s}Q$.
\vskip 0.02in

\subsection{Quasinormalizers of groups and algebras}Given groups $\Om\leqslant \G$, the \emph{one-side quasinormalizer semigroup} $QN^{(1)}_{\G}(\Om)\subseteq \G$ is the set of all $\g\in \G$ for which there is a finite set $F\subseteq \G$ so that $\Om\g\subseteq F\Om$, \cite[Section 5]{JGS10}; equivalently, $\g\in QN^{(1)}_{\G}(\Om)$ iff $[\Omega:\g\Omega\g^{-1}\cap\Omega]<\infty$. Thus  $QN^{(1)}_\G(\Om)$ coincides with the one-side commensurator of $\Om$ in $\G$.
\vskip 0.05in
\noindent Similarly, the \emph{quasinormalizer} (also called the \emph{commensurator}) $QN_\G(\Om)$ is the set of all $\g \in \G$ for which there exists a finite set $F\subseteq \G$ such that $\Om\g\subseteq F\Om$ and $\g\Om\subseteq \Om F$; 
equivalently, $\g\in QN_\G(\Om)$ iff $[\Omega:\g\Omega\g^{-1}\cap\Omega]<\infty$ and $[\g\Omega\g^{-1} :\g\Omega\g^{-1}\cap\Omega]<\infty$.
From definitions one checks that $QN_\G(\Om)\leqslant \G$ is a subgroup satisfying $\Omega\subseteq QN_\G(\Om)\subseteq QN^{(1)}_\G(\Om)$. 
\vskip 0.07in 
\noindent The von Neumann algebraic counterparts of (one-sided) quasinormalizers have played a major role in the recent classification results in this area \cite{Po99,Po01,Po03,IPP05}. Given an inclusion $Q \subseteq M$, the quasi-normalizer $QN_M(Q)$ is the $*$-algebra of all elements $x\in M$ such that there exist $x_1,x_2,...,x_k\in M$ so that  $Q x\subseteq \sum_i x_i Q$ and $x Q \subseteq \sum_i Q x_i$, \cite{Po99}. The von Neumann algebra $QN_M (Q)''$ is called the \emph{quasi-normalizing algebra of $Q$ inside $M$}. Similarly, the intertwiner space $QN^{(1)}_M(Q)$ is the set of all $x\in M$ such that there exist $x_1,x_2,...,x_k\in M$ so that  $Q x\subseteq \sum_i x_i Q$, \cite{Po99,JGS10}. The von Neumann algebra $QN^{(1)}_M(Q)''$ is called the \emph{one-sided quasinormalizer algebra of $Q$ inside $M$}.  
\vskip 0.05in
\noindent As usual $N_M(Q)=\{u\in \mathcal U(M) \,|\, uQu^*=Q \}$ denotes the \emph{normalizing group} and $ N_M(Q)''$ denotes the \emph{normalizing algebra of $Q$ in $M$}. Notice that $Q\subseteq  N_M(Q)\subseteq QN_M(Q)\subseteq QN_M^{(1)}(Q)\subseteq M$. 
\vskip 0.05in
\noindent The following computations of (one-sided) quasinormalizer algebras for inclusions of group von Neumann algebras will be essential for our arguments.
\vskip 0.01in
\begin{theorem}[Corollary 5.2 \cite{JGS10}]\label{iqncalc} If $\Om\leqslant \G$ is an inclusion of groups, the following hold:
\begin{enumerate}
\item $QN_{L(\G)}(L(\Om))'' = L(H_1)$, where $H_1= QN_\G(\Om)=QN^{(1)}_\G(\Om)\cap QN^{(1)}_\G(\Om)^{-1}$;
\item $QN^{(1)}_{L(\G)}(L(\Om))'' = L(H_2)$, where $H_2= \langle QN^{(1)}_\G(\Om)\rangle \leqslant \G$.
\end{enumerate}  
\end{theorem}
\begin{theorem}[Theorem \cite{Po03}; Proposition 6.2 \cite{JGS10}]\label{intcorner} Let $Q\subseteq M$ be an inclusion of tracial von Neumann algebras. Then for any $p\in \mathcal P(Q)$ we have  \begin{enumerate}\item $QN_{p M p}(p Q p)''= p QN_{M}(Q)'' p$, and
\item $QN^{(1)}_{p M p}(p Q p)''= p QN^{(1)}_{M}(Q)'' p$.
\end{enumerate}
\end{theorem}
 \vskip 0.01in

\subsection{Height of elements in group von Neumann algebras} Following \cite[Section 4]{Io10} and \cite[Section 3]{IPV10} the height $h_\G(x)$ of an element $x\in L(\G)$ is absolute value of the largest Fourier coefficient, i.e.,
$h_\G(x) =\sup_{\g\in \G}|\tau(xu_\g)|$, where $\{u_\g |\g\in \G\}$ are the canonical unitaries of $M$ implemented by $\G$. For a subset $S\subseteq L(\G)$, we denote by $h_\G(S) = \inf_{x \in S} h_\G(x)$. In this section we prove two elementary lemmas on height that will be used in the proof Theorem \ref{main4}. 

\begin{lem}\label{perturbationheight1} Assume that $L(\G)=M$ and consider the subsets  $\Sg\subseteq \G$ and $\mathcal S \subseteq (M)_1$. If there exist $x,y \in M$ such that $h_\Sg(x\mathcal S y)>0$ then one can find a finite subset $F\subset \La$ such that $h_{F \Sg F} ( \mathcal S )>0$. In particular,  we have $h_\La(\mathcal S)>0$ iff $h_\La(u \mathcal S u^*)>0$ for some $u\in \mathcal U(M)$. \end{lem}
\begin{proof} Using Kaplansky's Theorem for every $\varepsilon>0$ there is a finite subset $F_\varepsilon\subset \La$ and $x_\varepsilon,y_\varepsilon \in (M)_1$ supported on $F_\varepsilon$ so that $\|x-x_\varepsilon\|_2, \|y-y_\varepsilon\|_2\leq \varepsilon$. Using these estimates together with triangle inequality, for every $s\in \mathcal S$ and $\g\in \Sg$ we have     
\begin{equation*}\begin{split}
&|\tau (x s y u_{\g}) | \leq 4\varepsilon + | \tau (x_\varepsilon s y_\varepsilon u_{\g} )|  \\ & \leq 4\varepsilon + \sum_{\la,\mu \in F_\varepsilon }  |\tau(x_\varepsilon  u_{\la^{-1}}  )| |\tau(y_\varepsilon  u_{\mu^{-1}})| | \tau ( u_\la s  u_\mu u_{\g})| \leq 4\varepsilon +  |F_\varepsilon|^2\max_{\nu \in F_\varepsilon \g F_\varepsilon}  |\tau(s u_\nu)|  . 
\end{split} 
\end{equation*}
This implies that $h_{\Sg}(x \mathcal S y )\leq 4\varepsilon+  |F_\varepsilon|^2 h_{F_\varepsilon \Sg F_\varepsilon}(\mathcal S ) $ and hence $h_{F_\varepsilon \Sg F_\varepsilon}(\mathcal S )\geq (h_{\Sg}(x\mathcal S y ) - 4\varepsilon)|F_\varepsilon|^{-2}$. Letting $\varepsilon>0$ small enough we get the first part conclusion. 

For the remaining part, the reverse implication follows from above. The direct implication follows from the reverse implication by replacing $\mathcal S$ with $u\mathcal Su^*$ and $u$ with $u^*$. \end{proof}

\begin{lem}\label{perturbationheight2} Assume that $L(\G)=L(\La)=M$. Consider the comultiplication along $\La$ i.e. the embedding $\Delta : M \ra M\bar\otimes M$  given  by $\Delta(v_\la )=v_\la\otimes v_\la$, where $\{v_\la \,|\, \la\in \La\}$ are the canonical group unitaries of $M$ implemented by $\La$.  If $\mathcal S\subseteq (M)_1$ and there are $x,y\in (M\bar\otimes M)_1$ so that $h_{\G\times \G}( x\Delta (\mathcal S) y)>0$ then $h_\La(\mathcal S)>0$.
\end{lem}

\begin{proof} Since $h_{\G\times \G}( x\Delta (\mathcal S) y)>0$ then Lemma \ref{perturbationheight1} implies that \begin{equation}\label{501}h_{\G\times \G}( \Delta (\mathcal S) )>0.\end{equation}  Fix $s\in \mathcal S$. Let $s=\sum_{\lambda\in \Lambda}\tau(s v_{\lambda^{-1}})v_\lambda$ and note that $\Delta(s)=\sum_{\lambda\in \Lambda}\tau(s v_{\lambda^{-1}}) v_\lambda\otimes v_\lambda.$ Using these formulas and Cauchy-Schwarz inequality, for any $\g_1,\g_2\in \G$,  we have     
\begin{equation*}\begin{split}
&|\tau\otimes \tau (\Delta(s) (u_{\g_1} \otimes u_{\g_2})) | \leq \sum_{ \la\in \La  }   |\tau( s v_{\la^{-1}})| |\tau( v_{\la } u_{\g_1})| |\tau ( v_{ \la }  u_{\g_2})|\\ & \leq  h_{\La}(s)   \sum_{\la\in \La  }    |\tau( v_{\la } u_{\g_1})| |\tau ( v_{ \la }  u_{\g_2})| \leq   h_{\La}(s)    \|u_{\g_1}\|_2 \|u_{\g_2}\|_2= h_{\La}(s). 
\end{split} 
\end{equation*}

This further implies $h_{\G\times \G}(\Delta(\mathcal S)  )\leq  h_{\La}(\mathcal S) $ and using \eqref{501} we get the conclusion. \end{proof}

\subsection{Relative amenability} A tracial von Neumann algebra $(M,\tau)$ is called amenable if there exists a state $\phi:B(L^2(M))\ra \mathbb C$ such that $\phi_{|M} = \tau$ and  $\phi$ is $M$-central (i.e.\ $\phi(xT) = \phi(Tx)$ for all $x \in M, T \in B(L^2(M))$), \cite{Co76}. Making use of the basic construction for inclusions of algebras \cite{Ch79,Jo81} this concept was further generalized in \cite{OP07} to subalgebras.  Let $(M,\tau)$ be a tracial von Neumann algebra, $p\in M$ be a projection, and $P\subseteq pMp,Q\subseteq M$ be von Neumann subalgebras. Following \cite[Section 2.2]{OP07} we say that $P$ {\it is  amenable relative to $Q$ inside $M$} if there exists a $P$-central state $\phi:p\langle M,e_Q\rangle p\rightarrow\mathbb C$ such that  $\phi(x)=\tau(x)$, for all $x\in pMp$. Here $\langle M,e_Q\rangle$ denotes the basic construction for the inclusion $Q\subseteq M$, i.e.\ the commutant of the $Q$-right action on $B(L^2(M))$ \cite{Jo81}. 
\vskip 0.03in
In this section we prove a relative amenability result for subalgebras that ``cluster at infinity'' in an infinite tensor product of factors. The result will be essentially used to derive our infinite product rigidity result for group factors. Our proof is an adaptation of an argument due to Ioana. See also \cite[Lemma 4.4]{Is16} and \cite[Proposition 4.2]{HU15} for similar results. 
\begin{prop}\label{relamenlemma} Let $M\subseteq (\tilde M,\tau)$ be finite von Neumann algebras satisfying the following: \begin{enumerate}\item there are subalgebras $B,C \subseteq M$ and  $C\subseteq \tilde C\subseteq \tilde M$ so that  $M=B\vee C$ and $\tilde M=B\vee \tilde C $;
\item there are descending family $B_n\subseteq M$ of subalgebras and an ascending family $C_n\subseteq M$ of subalgebras such that $\cup_n C_n=C$, $\cap_n B_n=B$ and $M=B_j\vee C_j$ for all $j$. 
\item there exist $\theta, \theta_n \in Aut (\tilde M)$ such that ${\theta_n}_{|B_n}= {\rm id}_{|B_n}$ for all $n$,   $\theta_n\ra \theta$ pointwise,\end{enumerate} 
Let $p\in B$ be a nonzero projection and let $A\subseteq pMp$ be a von Neumann subalgebra so that for each $n\in\mathbb N$ there is $u_n\in \mathcal U( pMp )$ satisfying $u_nAu_n^*\subseteq B_n$.  Consider the $M$-$M$ bimodule  $\mathcal H=L^2(\tilde M)$ given by the actions $x\cdot \xi \cdot y= x\xi\theta(y)$ for all $x,y\in M$ and $\xi\in L^2(\tilde M)$. Then there exists a sequence of vectors $(\xi_n)_n \subset  L^2(p\tilde Mp)$ satisfying
\begin{eqnarray}
&&\lim_n\|x \cdot \xi_n-\xi_n \cdot x\|_{\mathcal H}=0,\text{ for all } x\in A, \text{ and}\\ 
&&\lim_n \langle x\cdot \xi_n,\xi_n\rangle_{\mathcal H}=\tau(x), \text{ for all } x\in pMp.
\end{eqnarray}   
\end{prop}
\begin{proof} Since $u_nAu_n^*\subseteq B_n$ and ${\theta_n}_{|B_n}= {\rm id}_{|B_n}$ then for every $x\in A$ we have $\theta_n(u_n xu_n^*)=u_n xu_n^*$. This implies that $u_n^*\theta_n(u_n) \theta_n(x)=xu_n^*\theta_n(u_n)$ and letting $\xi_n=u_n^*\theta_n(u_n)\in \mathcal U(p\tilde Mp)$ we conclude that for all $x\in A$ and $n\in \mathbb N$ we have \begin{equation}\label{101}
\xi_n \theta_n(x)=x \xi_n.
\end{equation}
Since $u_n\in \mathcal U(pMp)$ and ${\theta_n}_{|B_n}= {\rm id}_{|B_n}$ we have that $\|\xi_n\|_{\mathcal H}= \| u_n^*\theta_n(u_n)\|_2=\|p\|_2$ for all $n$. Since $\|u_n\|_\infty\leq 1$ then using \eqref{101} and $\theta_n \ra \theta$ pointwise one can check that for every $x\in A$ we have  
\begin{equation*}\label{102}\begin{split}
\lim_n \|x \cdot \xi_n -\xi_n \cdot x\|_{\mathcal H} &= \lim_n \|x \xi_n -\xi_n \theta(x)\|_{2}=\lim_n \|\xi_n (\theta_n(x) - \theta(x))\|_{2}\\&\leq \lim_n\|\xi_n\|_\infty \|\theta_n(x)-\theta(x)\|_2  =\lim_n \|\theta_n(x)-\theta(x)\|_2=0.  \end{split} 
\end{equation*}
Finally, since $\xi_n\in \mathcal U(p\tilde Mp)$ we have $\langle x\cdot \xi_n,\xi_n\rangle_{\mathcal H}= \tau(\xi_n^* x\xi_n)= \tau(x\xi_n\xi_n^* )=\tau(x)$ for all $x\in pMp$. Altogether, the above relations give the desired conclusion.
\end{proof}
\begin{prop}\label{relamenable} Let $M= \bar\otimes_{i\in \mathbb N} M_i\bar\otimes B$. Let $A\subseteq M$ be a von Neumann algebra for which there exist sequences $(k_n)_n\subseteq \mathbb N$ and $(u_n)_n \subset \mathcal U (M)$  such that $k_n\nearrow \infty$ and $u_nAu_n^*\subseteq \bar\otimes_{i\geq k_n} M_i\bar\otimes B$ for all $n$.  Then $A$ is amenable relative to $B$ inside $M$.
\end{prop}
\begin{proof} Denote by $\bar\otimes_{i\in \mathbb N} M_i=C$ and notice that $M=C\bar\otimes B$. Let $\tilde C=C\bar\otimes C$ and $\tilde M=\tilde C\bar\otimes B$ and notice that $M\subset \tilde M$. For every $n\in \mathbb N$ denote by $C_n= \bar\otimes^{k_n-1}_{i= 1} M_i$, by $D_n=\bar\otimes_{i\geq k_n} M_i$ and by $B_n=D_n\bar\otimes B$ and notice that $\cup_n C_n= C$ and $\cap_n B_n=B$.  
 Next let $\theta_n \in Aut(\tilde M)$ satisfying $\theta_n (x\otimes y)=y\otimes x$ for all  $x,y \in  C_n$ and $\theta_n ={\rm id}$  on $D_n \bar\otimes D_n \bar\otimes B$. Notice that $\theta_n \ra \theta$ pointwise, where $\theta\in Aut (\tilde M)$ satisfies $\theta (x\otimes y)=y\otimes x$ for all  $x,y \in  C$ and $\theta ={\rm id}$  on $ B$. One can check all the conditions in the statement of Proposition \ref{relamenlemma} are satisfied. Thus if we consider the $M$-$M$ bimodule $\mathcal H:= L^2(\tilde M)= L^2(C)\bar\otimes L^2(C)\bar\otimes L^2(B)$ with the actions given by $x\cdot \xi\cdot y= x\xi\theta(y)$ for all $x,y\in M$ and $\xi\in \mathcal H$ there exist a sequence of unit vectors $(\xi_n)_n\in \mathcal H$ such that
\begin{equation}\label{110}\begin{split}
&\lim_n\|x \cdot \xi_n-\xi_n \cdot x\|_{\mathcal H}=0,\text{ for all } x\in A\\ 
&\lim_n \langle x\cdot \xi_n,\xi_n\rangle_{\mathcal H}=\lim_n \langle \xi_n \cdot x,\xi_n\rangle_{\mathcal H}=\tau(x), \text{ for all } x\in M.
\end{split}\end{equation}
Let $\generator{M, e_{1\otimes B}}$ be the basic construction for $1\bar\otimes B\subset M$ and let $Tr$ be the semifinite trace on $\generator{M, e_{1\otimes B}}$. Next we notice that, as $M$-$M$-bimodules, $ L^2(\generator{ M, e_{1\otimes B}}, Tr)$ is isomorphic to $\mathcal{H}$ via the map $(x\otimes y)e_{1\otimes B} (z\otimes 1)\ra (x\otimes y)\cdot (1\otimes 1\otimes 1) \cdot  (z\otimes 1) $, for $x,z\in C$ and $y\in B$. Indeed it is clear this is $M$-$M$-bimodular and also for all $x_i,z_i \in C$ and $y_i\in B$ we have 
\begin{align*}
&\generator{(x_1\otimes y_1) e_{1\otimes B} (z_1\otimes 1), (x_2\otimes y_2) e_{1\otimes B} (z_2\otimes 1)}_{Tr} =\\
=&\operatorname{Tr} (  (z_2\otimes 1)^*e_{1\otimes B} (x_2\otimes y_2)^* (x_1\otimes y_1) e_{1\otimes B} (z_1\otimes 1)    )\\
	=& \operatorname{Tr} ( (z_1z_2^*\otimes 1) E_{1\otimes B}(x_2^*x_1\otimes y_2^*y_1  )e_{1\otimes B}  )\\
	=&\tau_C(x_2^*x_1) \tau_{C\otimes B}(z_1z_2^*\otimes y_2^*y_1 )\\
	=&\tau_C( x_2^*x_1 )\tau_{C}(z_2^*z_1) \tau_B( y_2^*y_1 )\\
	=&\generator{ (x_1\otimes 1 \otimes y_1) (1\otimes 1\otimes 1)\theta( z_1\otimes 1), (x_2\otimes 1 \otimes y_2)( 1\otimes 1\otimes 1) \theta( z_2\otimes 1) )}_{\mathcal H}\\
	=&\generator{ (x_1\otimes y_1)\cdot (1\otimes 1\otimes 1)\cdot (z_1\otimes 1), (x_2\otimes y_2)\cdot (1\otimes 1\otimes 1) \cdot (z_2\otimes 1)}_{\mathcal H}.
\end{align*} 
This combined with \eqref{110} and \cite[Theorem 2.1]{OP07} show that $A$ is amenable relative to $B$ inside $M$. 
\end{proof}


\section{Proof of Theorem \ref{main1}}
This section is devoted to the proof of Theorem \ref{main1}. In essence this result is an infinite analog of the ``product rigidity'' phenomenon for group factors found in \cite{CdSS16}. In fact our methods build upon the general strategy developed in \cite{CdSS16} and still use in a crucial way the ultrapower techniques from \cite{Io11} as well as the intertwining/combinatorial aspects developed in \cite{OP03,IPV10,CdSS16,DHI16,CdSS17} and the classification of normalizers from \cite{PV12}. Since our exposition will focus primarily on the novel aspects of these techniques we recommend the reader to consult the aforementioned works as some of these results will be heavily used throughout the section.    
\vskip 0.05in
To ease our exposition we first introduce the following notation:  
\begin{notation}\label{directprod} 
 Let $\{\Gamma_i\}_{ i\in I}$ be a collection of icc, weakly amenable, biexact groups and denote by $\Gamma=\oplus_{i\in I} \Gamma_i $. For any subset $S\subseteq I $, we denote $\Gamma_{S}=\oplus_{i\in S} \Gamma_i $.
Denote by $M=L(\Gamma)$, let $t>0$ be a scalar, and assume that  $M^t=L(\Lambda) $ for an arbitrary group $\La$.  Following \cite{IPV10}, let $\Delta: M^t\to M^t\bar\otimes M^t $ be the commultiplication along $ \Lambda$, i.e.\ $\Delta(v_\lambda)=v_\lambda\otimes v_\lambda$, where $\{v_\lambda\}_{\lambda\in \Lambda} $ are the canonical unitaries generating $L(\Lambda)$.  
\end{notation}

\begin{prop}\label{firstintertwining} Assume Notation \ref{directprod}. Then for every $i\in I $ there exists $j\in I$ such that $\Delta(L(\Gamma_{I\setminus \{i\}})^t)\prec_{M^t\bar\otimes M^t} M^t\bar\otimes L(\Gamma_{I\setminus\set{j}})^t$. 
\end{prop}
\begin{proof} Using that $\G_i$'s are weakly amenable and biexact we show next the following

\begin{claim}\label{prop:IntertwineDichotomy} For every $i,j\in I$ one of the following holds: 
\begin{enumerate}
\item [a)] $\Delta(L(\Gamma_{I\setminus\set{i}})^t)\prec_{M^t\bar\otimes M^t} M^t\bar\otimes L(\Gamma_{I\setminus\set{j}})^t $, or
\item [b)] $\Delta(L(\Gamma_{i}))\prec_{M^t\bar\otimes M^t} M^t\bar\otimes L(\Gamma_{I\setminus\set{j}})^t $.
\end{enumerate}
\end{claim}

\noindent\emph{Proof of Claim \ref{prop:IntertwineDichotomy}}. One has the following decomposition 
$M^t\bar\otimes M^t=M^t\bar\otimes  L(\Gamma_{I\setminus\set{j}})^t\bar\otimes L(\Gamma_j).$
Fix $A\subset \Delta(L(\Gamma_i)) $ a diffuse amenable subalgebra.  Using \cite[Theorem 1.6]{PV12},  we have either
\begin{enumerate}
\item [c)] $A\prec_{M^t\bar\otimes M^t} M^t\bar\otimes L(\Gamma_{I\setminus\set{j}})^t $, or\label{case:DichotomyStep1Intertwine}
\item [d)] \label{case:DichotomyStep1Normal} $\mathcal{N}_{M^t\bar\otimes M^t}(A)''$ is amenable relative to  $M^t\bar\otimes L(\Gamma_{I\setminus\set{j}})^t $. 
\end{enumerate}
Suppose d) holds.  As $\Delta(L(\Gamma_{I\setminus\set{i}})^t)\subseteq \mathcal{N}_{M^t\bar\otimes M^t}(A)'' $, then \cite[Theorem 1.6]{PV12} implies either
\begin{enumerate}[resume]
\item [e)] $\Delta(L(\Gamma_{I\setminus\set{i}})^t)\prec  M^t\bar\otimes L(\Gamma_{I\setminus\set{j}})^t$, or\label{case:DichotmyStep2Intertwine}
\item [f)] $\mathcal{N}_{M^t\bar\otimes M^t}\Delta( L(\Gamma_{I\setminus\set{i}})^t )''$ is amenable relative to $M^t\bar\otimes L(\Gamma_{I\setminus\set{j}})^t $.\label{case:DichotomyStep2Normal}
\end{enumerate}
However, f) cannot hold.  Indeed, since $\Delta(M^t)\subseteq \mathcal{N}_{M^t\bar\otimes M^t} (\Delta(L(\Gamma_{I\setminus\set{i}}))^t)'' $, then \cite[Theorem 7.2(2)]{IPV10} would imply $\Gamma_i $ is finite, a contradiction.  Hence, for every diffuse subalgebra $A\subset L(\Gamma_i) $, either c) or e) must occur.  Using \cite[Appendix]{BO08}, we get the claim.$\hfill\blacksquare$
\vskip 0.04in

Now assume by contradiction the conclusion does not hold. By Claim \ref{prop:IntertwineDichotomy}, for every $j\in I$ we have
\begin{align}\label{eq:IntertwineAllMinus1}
 \Delta(L(\Gamma_i))\prec_{M^t\bar\otimes M^t} M^t\bar\otimes L(\Gamma_{I\setminus\set{j}})^t.
\end{align}
Next we observe that $\mathcal{Z}(\Delta(L(\Gamma_i))'\cap M^t\bar\otimes M^t  )=\C 1 $.  To see this, let $z\in \mathcal{Z}(\Delta(L(\Gamma_i))'\cap M^t\bar\otimes M^t  ) $.  Since $\Delta(L(\Gamma_{I\setminus \{i\}})^t)\subset \Delta(L(\Gamma_i))'\cap M^t\bar\otimes M^t $, one can check that $z\in \Delta(M^t)'\cap M^t\bar\otimes M^t $. However, since $\La$ is icc we have  $\Delta(M^t)'\cap M^t\bar\otimes M^t=\mathbb C 1 $ and our claim follows. 

Thus \eqref{eq:IntertwineAllMinus1}  further implies that $\Delta(L(\Gamma_i))\prec^s_{M^t\bar\otimes M^t} M\bar\otimes L(\Gamma_{I\setminus\set{j}})$. Hence, applying \cite[Lemma 2.8 (2)]{DHI16}, for every finite subset $F\subset I $ we have 
\begin{align}\label{eq:IntertwineFiniteStronly}
 \Delta(L(\Gamma_i))\prec_{M^t\bar\otimes M^t} M^t\bar\otimes L(\Gamma_{I\setminus F})^t.
\end{align}
 Next we show that (\ref{eq:IntertwineFiniteStronly})  implies the following   
\begin{claim}\label{relamen}$\Delta(L(\Gamma_i)) $ is amenable relative to $ M^t\otimes 1$.
\end{claim}
\noindent \emph{Proof of Claim \ref{relamen}}. Let $I_n= \{n,n+1,n+2,...\}$.  Since $\Delta(L(\Gamma_{i}))'\cap M^t\bar\otimes M^t $ is a factor, then using \cite[Proposition 12]{OP03}, for every $n\in \mathbb N$ there is $t_n>0 $ and $u_n\in \mathcal{U}(M^t\bar\otimes M^t) $ so that
\begin{align*}
u_n \Delta(L(\Gamma_i)) u_n^*\subset (M^t\bar\otimes L(\Gamma_{I_n})^t)^{t_n}.
\end{align*}
Naturally, we have the following inclusions $M^t\bar\otimes L(\Gamma_{I_n})^{tt_n} \subset M^t\bar\otimes L(\Gamma_{I_n})\bar\otimes L(\Gamma_{n-1})=M^t\bar\otimes L(\Gamma_{I_{n-1}})$.   Thus, for every $n\in I$, there is $u_n\in\mathcal{U}(M^t\bar\otimes M^t) $ so that 
\begin{align}\label{eq:ConjugateFiniteIntoMax}
u_n\Delta(L(\Gamma_i))u_n^*\subset M^t\bar\otimes L(\Gamma_{I_{n-1}}).
\end{align}

\noindent Thus the claim follows from \eqref{eq:ConjugateFiniteIntoMax} and Proposition \ref{relamenable}. $\hfill\blacksquare$
\vskip 0.04in
\noindent Finally, Claim \ref{relamen} and \cite[Proposition 7.2]{IPV10} imply $L(\Gamma_{i}) $ amenable, a contradiction.\end{proof}

\begin{prop}\label{IntIntoGroup1} Assume Notation \ref{directprod}. Then for all $i\in I $, there exists a non-amenable subgroup $\Lambda_i\leqslant \Lambda $ with non-amenable centralizer $C_\Lambda(\Lambda_i) $ such that $L(\Gamma_{I\setminus\set{i}})^t\prec_{M^t} L(\Lambda_i).$
\end{prop}
\begin{proof}
This follows directly from Proposition \ref{firstintertwining} and \cite[Theorem 4.1]{DHI16}, (see also the proof of \cite[Theorem 3.3]{CdSS16}).
\end{proof}

\begin{theorem}\label{TwoFoldedProduct}Assume Notation \ref{directprod}. In addition, assume that $\G_i$ has property (T), for all $i\in I$. For each $i\in I$ there is a decomposition $\La=\Psi_i\oplus \Theta_i$, a scalar $t_i>0$ and $u_i\in \mathcal U(M)$ satisfying \begin{equation}u_iL(\G_i)^{t_i}u_i^*=L(\Psi_i)\text{ and  }u_iL(\G_{I\setminus\{i\}})^{t/t_i}u_i^*=L(\Theta_i).\end{equation} 

\end{theorem}

\begin{proof}Fix $i\in I $ and write $M^t=L(\Gamma_I)^t=L(\Gamma_{I\setminus\set{i}})^t\bar\otimes L(\Gamma_i)=A\bar\otimes B $. By Proposition \ref{IntIntoGroup1}, we have $A\prec_{M^t} L(\Lambda_i)$ for some non-amenable group $\Lambda_i\leqslant \Lambda $ with non-amenable $C_{\La}(\La_i)$.  By \cite[Proposition 2.4]{CKP14}, there exist nonzero projections $a\in A, q\in L(\Lambda_i) $, a partial isometry $v\in M^t$, a subalgebra $D\subseteq qL(\Lambda_i) q$, and a $\ast $-isomorphism $\phi:aAa\to D  $ such that 
\begin{align}
D\vee (D'\cap qL(\Lambda_i) q)\subseteq qL(\Lambda_i) q\quad \text{has finite index, and}\label{eq:DHasFiniteIndex}\\
\phi(x)v=vx \quad \forall \, x\in aAa.\label{eq:PartialIsomIntertwinesA}
\end{align}
Notice that $vv^*\in D'\cap qM^tq $ and $v^*v\in (aAa)'\cap aM^ta=a\otimes B $.  Hence there is a projection $b\in B$ satisfying $v^*v=a\otimes b $.  Picking $u\in\mathcal{U}(M^t) $ so that $v=u(a\otimes b) $ then \eqref{eq:PartialIsomIntertwinesA} gives 
\begin{align}\label{cornerbegin}
Dvv^*=vaAav^*=u(aAa\otimes b)u^*.
\end{align}
Passing to the relative commutants, we obtain $vv^*(D'\cap qM^tq) vv^*=u(a\otimes bBb) u^*$. This further implies that there exist $s_1,s_2> 0 $ \begin{align}\label{corner}(D'\cap qM^tq)z=u(a\otimes bBb)^{s_1}u^*= L(\G_i)^{s_2},
\end{align}
where $z $ is the central support projection of $ vv^*$ in $D'\cap qMq$.   Now notice 
\begin{align*}
D'\cap qM^tq\supseteq (qL(\Lambda_i)q)'\cap qM^tq= (L(\Lambda_i)'\cap M^t)q\supseteq L(C_\Lambda(\Lambda_i))q,
\end{align*}
where $L(C_\Lambda(\Lambda_i)) $ has no amenable direct summand since $C_\Lambda(\Lambda_i) $ is a non-amenable group.  Moreover we also have  $D'\cap qM^tq\supseteq D'\cap qL(\Lambda_i)q $. Thus $(L(\Lambda_i)'\cap M^t)z$ and $(D'\cap qL(\Lambda_i)q)z$ are commuting subalgebras of $(D'\cap qM^tq)z$ where $(L(\Lambda_i)'\cap M^t )z$ has no amenable direct summand.  Since $\Gamma_i $ was assumed to be bi-exact, then using \eqref{corner} and \cite[Theorem 1]{Oz03} it follows that $(D'\cap qL(\Lambda_i)q) z$ is purely atomic.  Thus, cutting by a central projection $r'\in D'\cap q(\Lambda_i)q$ and using \eqref{eq:DHasFiniteIndex} we may assume that $D\subseteq qL(\Lambda_i)q$ is a finite index inclusion of algebras.  
Processing as in the second part of \cite[Claim 4.4]{CdSS16}, we may assume that $D\subseteq qL(\Lambda_i)q $ is a finite index of II$_1$ factors. Moreover one can check that if one replaces $v$ by the partial isometry of the polar decomposition of $r'v\neq 0$ then all relations \eqref{eq:PartialIsomIntertwinesA},\eqref{cornerbegin} and \eqref{corner} are still satisfied. In addition, we can assume without any loss of generality that the support projection satisfies $s(E_{L(\La_i)}(vv^*)=q$. Thus, following the terminology introduced in \cite[Definition 4.1]{CdSS17}  we actually have that a corner of $A$ is spatially commensurable to a corner of $L(\La_i)$, i.e.\begin{equation}\label{commensurability}
A\cong^{com}_{M^t} L(\La_i).
\end{equation}

Performing the downward basic construction \cite[Lemma 3.1.8]{Jo81}, there exists  $e\in \mathscr P(qL(\Lambda_i) q)$ and a II$_1$ subfactor $R\subseteq D\subseteq qL(\La_i) q=\langle D,e\rangle$ such that $[D:R]=[qL(\Lambda_i) q:D]$ and $Re=eL(\Lambda_i) e$. Keeping with the same notation, by relation (\ref{eq:PartialIsomIntertwinesA}) the restriction $\phi^{-1}: R\ra aAa$ is  an injective $\ast$-homomorphism such that $T=\phi^{-1}(R)\subseteq aAa$ is a finite Jones index subfactor and 
\begin{equation}\label{4}
\phi^{-1} (y)v^*=v^*y\text{, for all }y\in R.
\end{equation}  
Let $\theta':Re\ra R$ be the $\ast$-isomorphism given by $\theta(xe)=x$. Since $e$ has full central support in $\langle D,e\rangle$ one can see that $ev\neq 0$. Letting $w_0$ be a partial isometry so that $w^*_0|v^*e|=v^*e$, then $Re=eL(\Lambda_i) e$ together with (\ref{4}) imply that $\theta=\phi^{-1}\circ\theta': eL(\Lambda_i) e\ra aAa$ is an injective $\ast$-homomorphism satisfying $\theta(eL(\Lambda_i )e)=T$ and  
\begin{equation}\label{3}
\theta (y)w^*_0=w^*_0y\text{, for all }y\in eL(\Lambda_i) e.
\end{equation} 
Notice that $w^*_0w_0\in(T'\cap aAa)\bar\otimes  B$ and proceeding as in the proof of \cite[Proposition 12]{OP03} one can further assume that $w^*_0w_0\in\mathscr Z(T' \cap aAa) \bar\otimes B$. Since $[aAa:T]<\infty$ then $T' \cap aAa$ is finite dimensional and so is $\mathscr Z(T' \cap aAa)$. Thus, replacing the partial isometry $w_0$ by $w:=w_0r_0$, for some minimal projection  $r_0\in  \mathscr Z(T' \cap aAa)$ satisfying $r_0w^*_0|v^*e|\neq 0$, we see that all relations above still hold including relation (\ref{3}). Moreover, we can assume that $w^*w= z_1\otimes z_2$, for some nonzero projections $z_1\in   \mathscr Z(T' \cap aAa)$ and $z_2\in B$. Using relation (\ref{3}) we get  \begin{equation}\label{5}w^* L(\Lambda_i) w=\theta (eL(\Lambda_i) e)w^*w= Tz_1 \otimes z_2.
\end{equation}
Since $T\subseteq aAa$ is finite index inclusion of II$_1$ factors then by the local index formula \cite{Jo81} it follows $Tz_1\subseteq z_1Az_1$ is a finite index inclusion of II$_1$ factors as well. 
Also, we have 
 \begin{equation}\label{6}
( w^* L(\Lambda_i) w)' \cap (z_1\otimes z_2)M^t (z_1\otimes z_2)= ((Tz_1)'\cap z_1Az_1) \bar\otimes z_2Bz_2.
 \end{equation}
 
 Altogether, the previous relations imply that
 \begin{equation}\label{6'}\begin{split}
 Tz_1\bar \otimes z_2Bz_2&\subseteq Tz_1\vee (Tz_1'\cap z_1Az_1) \bar\otimes z_2Bz_2 \\&= w^*L(\Lambda_i) w\vee  w^*( L(\Lambda_i)'\cap M^t) w\\
 &=w^* L(\Lambda_i) w\vee \left((w^* L(\Lambda_i) w)' \cap (z_1\otimes z_2)M^t (z_1\otimes z_2)\right)\\
 &\subseteq z_1Az_1\bar\otimes z_2Bz_2.
 \end{split}
 \end{equation} 
Since $Tz_1\subseteq z_1Az_1$ if a finite index inclusion of II$_1$ factors then so is $Tz_1\bar\otimes z_2Bz_2\subseteq z_1Az_1\bar\otimes z_2Bz_2 $.  Let $f:=ww^*$ and notice $f=re$, for some projection $r\in L(\La_i)'\cap M^t$. Letting $u\in \mathscr U(M^t) $ such that $w^*= uww^*=uf$, then relation (\ref{6'}) further implies that \begin{equation}\label{finiteindex2}f(L(\Lambda_i) \vee (L(\Lambda_i)'\cap M^t)) f=L(\Lambda_i) f \vee  f(L(\Lambda_i) '\cap M^t)f\subseteq f M^t f
\end{equation}
is an inclusion of finite index II$_1$ factors. In addition, (\ref{6'}) gives that $\dim_\mathbb C (\mathscr Z(f(L(\Lambda_i) \vee (L(\Lambda_i)'\cap M^t)) f))\leq [ z_1Az_1\bar\otimes z_2Bz_2 : Tz_1\bar\otimes z_2Bz_2]<\infty$. Since the central support of $e$ in $qL(\La_i)q$ equals $q$ then \eqref{finiteindex2} implies that  
\begin{equation}\label{finiteindex3}q(L(\Lambda_i)qr \vee r(L(\Lambda_i)'\cap M^t))rq=qr(L(\Lambda_i) \vee (L(\Lambda_i)'\cap M^t)) qr  \subseteq qr M^t qr,
\end{equation}
is a finite index inclusion of II$_1$ factors. In particular, $ qL(\Lambda_i)qr$ and  $r(L(\Lambda_i)'\cap M^t)rq $ are commuting II$_1$ factors.

To this end we notice that since $0\neq r_0w^*_0|v^*e |=w^*|v^*e|$ then $0\neq w^*|v^*e|^{1/2}$. Thus $0\neq w^*|v^*e| w= v^*ew$ and since $v$, $w$ are partial isometries we conclude that $0\neq vv^* eww^*$. However since $ww^*= r_0w_0w_0^*\leq s(|v^*e|)$ then $ww^*\leq e$. Combining with the above it follows that $0\neq vv^*ww^*$ and hence $zf=z ww^*\neq 0$. Thus further implies that \begin{equation}\label{nonzeroproduct2}
zr\neq 0.
\end{equation}
 
\noindent Next we show the following 
\begin{claim}\label{propT} $r(L(\Lambda_i)'\cap M^t )rq$ has property (T).
\end{claim}
\noindent \emph{Proof of Claim \ref{propT}}. Since $D \subseteq qL(\La_i)q$ is a finite index inclusion of II$_1$ factors then so is   $Dr \vee r(L(\La_i)'\cap M^t)rq\subseteq  qL(\La_i)q r \vee r(L(\La_i)'\cap M^t) rq$. Using \eqref{finiteindex3} it follows that $Dr \vee r(L(\La_i)'\cap M^t)rq\subseteq rq  M^trq$ is finite index as well. Hence $Dr \vee r(L(\La_i)'\cap M^t)rq\subseteq  Dr \vee r(D'\cap qM^tq) r$ is also a finite index inclusion. Since $D$ is a factor one can check that  $E_{Dr \vee r(L(\La_i)'\cap M^t)rq}(x)= E_{ (L(\La_i)'\cap M^t)q}(x)$, for all $x\in r(D'\cap qM^tq)r$. This combined with the above entail that $r(L(\La_i)'\cap M^t)rq\subseteq  r(D'\cap qM^tq) r$ is finite index. By \cite[Theorem 1.1.2 (ii)]{Po94} $r(L(\La_i)'\cap M^t)rz\subseteq  r(D'\cap qM^tq) rz$ is finite index and since property (T) passes to amplifications and finite index subalgebras, then \eqref{corner} implies that $r(L(\Lambda_i)'\cap M^t )rz$ has property (T).  As $r(L(\Lambda_i)'\cap M^t )rq$ is a factor we conclude  $r(L(\Lambda_i)'\cap M^t )rq$ has property (T). $\hfill\blacksquare$
\vskip 0.04in
Now consider $\Omega:=vC_\Lambda(\La_i)=\set{\lambda\in \Lambda \,|\, |\lambda^{\Lambda_i}|<\infty} $, the virtual centralizer of $\Lambda_i $ in $\Lambda $. Using \cite[Claim 4.7]{CdSS16}  we have $[\Lambda: \Lambda_i\Omega]<\infty $ and hence $ \Lambda_i\Omega\leqslant \Lambda$  is an icc subgroup; in particular, $vZ(\Lambda_i\Omega)=1$. Consider $vZ(\Omega) =\{\omega\in\Omega \,|\, |\omega^\Omega|<\infty\}$, the virtual center of $\Omega$. Since $\La_i$ normalizes $\Omega$ one can check that $vZ(\Omega)\leqslant vZ(\La_i\Omega)$. Since the latter is trivial we get $vZ(\Omega)=1$ and hence $\Omega$ is icc.  
Let $(\mathcal O_n)_{n\in\mathbb N} $ be a countable enumeration of all the orbits under conjugation by $\La_i$. Denote by $\Omega_k=\langle\mathcal O_1,...,\mathcal{O}_k\rangle \leqslant \Omega$, the subgroup generated by $\mathcal{O}_n $, $n=\overline {1,k}$.  $\Omega_k $'s form an ascending sequence of subgroups normalized by $\La_i$ such that $\Omega=\cup_{k=1}^\infty\Omega_k $. Thus $\Lambda_i\Omega_k$ is an ascending sequence satisfying $\Lambda_i\Omega=\cup_{k=1}^\infty \Lambda_i\Omega_k $.  Since $r(L(\Lambda_i)'\cap M^t)rq \subset L(\La_i\Omega)$ has property (T) there is $k_0\in \mathbb N$  such that \begin{align}\label{eq:CommutantIntertwinesFiniteOrbit}
r(L(\Lambda_i)'\cap M^t)rq\prec_{L(\La_i\Omega)} L(\Lambda_i\Omega_{k_0}).
\end{align}
Next we show the following \begin{claim}\label{extendintertwining} There exists $k\geq k_0 $ such that $qL(\La_i)qr\vee r(L(\Lambda_i)'\cap M^t )rq\prec_{L(\La_i\Omega)} L(\Lambda_i\Omega_k)$. \end{claim}
\noindent\emph{Proof of the Claim \ref{extendintertwining}}. Using Popa's intertwining techniques, \eqref{eq:CommutantIntertwinesFiniteOrbit} implies the existence of  $x_\ell, y_\ell \in L(\Lambda_i\Omega) $, $ \ell=\overline{1,j}$ and $c>0 $ satisfying
\begin{align}\label{eq:IntertwineUnitaryBound}
\sum_{\ell=1}^j \norm{E_{L(\Lambda_i \Omega_{k_0})} (x_\ell u y_\ell) }_2\geq c,
\end{align}
for all $u\in\mathcal{U}(r(L(\Lambda_i)'\cap M^t)rq) $.  Since $\Lambda_i\Omega_k \nearrow\La_i\Omega$  for every $\varepsilon>0 $ there is  $k\geq k_0 $ so that 
\begin{align}\label{eq:ConditionalExpectationCloseToConstants}
\sum_{\ell=1}^j\norm{E_{ L(\Lambda_i\Omega_k) } (x_\ell)-x_\ell  }_2<\varepsilon,\quad \sum_{\ell=1}^j\norm{ E_{L(\Lambda_i\Omega_k)}(y_\ell)-y_\ell  }<\varepsilon.
\end{align}
Using \eqref{eq:IntertwineUnitaryBound} together with inequalities $\|mzn\|_2 \leq \|m\|_\infty \|z\|_2 \|n\|_\infty$ for all $m,n,z \in M^t$ then for all $u\in \mathcal{U}( r(L(\Lambda_i)'\cap M^t)rq ) $ we have
\begin{align*}
c\leq& \sum_{\ell=1}^j \norm{E_{L(\Lambda_i \Omega_{k})} (x_\ell u y_l) }_2\\
	\leq& \sum_{\ell=1}^j  \norm{ E_{L(\Lambda_i \Omega_{k})} ((x_\ell - E_{L(\Lambda_i \Omega_{k})} (x_\ell) ) uy_\ell)}_2+ \sum_{\ell=1}^j\norm{ E_{L(\Lambda_i \Omega_{k})}(E_{L(\Lambda_i \Omega_{k})}(x_\ell) u(y_\ell- E_{L(\La_i \Omega_k)}(y_\ell))}_2\\
	&+\sum_{\ell=1}^j\norm{ E_{L(\Lambda_i \Omega_{k})}(E_{L(\Lambda_i \Omega_{k})}(x_\ell)) uE_{L(\La_i \Omega_k)}(y_\ell)}_2\\
	\leq& \varepsilon \max_{1\leq \ell\leq j}( \norm{y_\ell}_\infty +\norm{x_\ell}_\infty)+\sum_{\ell=1}^j \norm{x_\ell}_\infty \norm{y_\ell}_\infty\norm{ E_{L(\Lambda_i \Omega_{k})}(u)  }_2\\
	\leq& 2 d\varepsilon  + jd^2\norm{ E_{L(\Lambda_i \Omega_{k})}(u) }_2,
\end{align*}
where $d:=\max_{1\leq \ell \leq j}\set{\norm{x_\ell}_\infty, \norm{y_\ell}_\infty} $. This shows there is $k\geq k_0 $ so that $\norm{E_{L(\Lambda_i \Omega_{k_0})}(u)}_2\geq \frac{c - 2\varepsilon d}{jd^2}$ for all $u\in\mathcal{U}(r(L(\Lambda_i)'\cap M^t)rq) $. Letting $\varepsilon=\frac{c}{4d}$ then for all $u\in \mathcal U(r(L(\Lambda_i)'\cap M^t)rq)$  we have 
\begin{align*}
\norm{E_{L(\Lambda_i \Omega_{k})}(u)}_2\geq \frac{c}{2jd^2}>0.
\end{align*}
 This implies for all $a\in \mathcal{U}(qL(\Lambda_i)qr)$ and $u\in \mathcal U(r(L(\Lambda_i)'\cap M^t)rq)$ we have  
\begin{align}\label{eq:UnitaryLambaIntertwines}
\norm{E_{L(\Lambda_i \Omega_{k})}(au)}_2=\norm{aE_{L(\Lambda_i \Omega_{k})}(u)}_2=\norm{E_{L(\Lambda_i \Omega_{k})}(u)}_2\geq \frac{c}{2jd^2}.
\end{align}
As $\mathcal{U}(qL(\Lambda_i)qr)\mathcal{U}(r(L(\Lambda_i)'\cap M^t)rq) $ generates $qL(\Lambda_i)qr\vee r(L(\Lambda_i)'\cap M)rq $, \eqref{eq:UnitaryLambaIntertwines} gives the claim. $\hfill\blacksquare$
\vskip 0.04in 
Now, since $q(L(\Lambda_i)qr\vee r(L(\Lambda_i)'\cap M)rq\subseteq rqL(\Lambda_i\Omega)rq $ is a  finite index inclusion, then $rq L(\Lambda_i\Omega)rq \prec_{L(\La_i\Omega)} (L(\Lambda_i \Omega_k)$ and hence $ L(\Lambda_i\Omega) \prec_{L(\La_i\Omega)} L(\Lambda_i \Omega_k)$. By \cite[Lemma 2.2]{CI17} it follows that $\La_i\Omega_k\leqslant \La_i\Omega$ has finite index and by increasing $k$ we can assume that $\La_i\Omega_k=\La_i\Omega$. Let $\La':=C_{\La_i}(\Omega_k)\leqslant \La_i$ and notice $[\La_i:\La']<\infty$. Thus $[\La_i\Omega:\La'\Omega_k]<\infty$ and since $\La_i\Omega$ is icc then $\La'\Omega_k$ is also icc. In particular, we also have $\La'\cap \Omega_k =1$. As $\La'\Omega_k\leqslant \La_i\Omega$ is finite index the  $\La'\Omega_k \cap \Omega \leqslant \Omega$ is also finite index. In particular since $\Omega$ is icc it follows that $\La'\Omega_k \cap \Omega$ is also icc. Letting $\La'': = \La'\cap \Omega$ the above considerations imply that 
$\La'' \Omega_k =\La'\Omega_k \cap \Omega$. This forces $\La''$ to be either trivial or icc. However, since by construction $\La''=vZ(\La'')$ then $\La''=1$. Since $\La'\leqslant \La_i$ finite index it follows that $\La_i$ is icc-by-finite and hence finite-by-icc.    
\vskip 0.03in  

This together with \eqref{commensurability} and \cite[Theorem 4.6]{CdSS17} show there exists $\Sg \leqslant C_\La(\La_i)$ such that $[\La: \La_i\Sg]<\infty$  and $B\cong^{com}_M L(\Sg)$. Also since $\La$ is icc then so are $\La_i$ and $\Sg$. Finally, using \cite[Theorem 4.7]{CdSS17} there is a decomposition $\La=\Psi_i\oplus \Theta_i$, $u_i\in \mathcal U(M^t)$ and $t>0$ such that $u_i A^tu_i^*=u_iL(\G_{I\setminus\{i\}} )^{t}u_i^*= L(\Psi_i)$ and $u_i B^{1/t}u_i^*=u_iL(\G_i )^{t}u_i^*= L(\Theta_i)$.\end{proof}

\begin{theorem}\label{infprodrig}
Let $(\G_n)_{n\in \mathbb N} $ a countable infinite collection of property (T), biexact, weakly amenable, icc groups. Assume that $\Lambda$ is an \emph{arbitrary} group satisfying $L(\oplus_{n} \G_n)=L(\Lambda)$. Then there exists an infinite direct sum decomposition $\La=(\oplus_n \La_n) \oplus A$ where $A$ is either trivial or icc amenable group. Moreover, for each $k\in \mathbb N$ there exist scalars $t_1,...,t_{k+1}>0$ satisfying $t_1 t_2 \cdots t_{k+1}=1$ and a unitary $u\in L(\La)$ so that 
\begin{equation}\label{intertwiningtheparts'}\begin{split}&uL(\G_n)^{t_n}u^*=L(\La_n)\quad \text{ for all }n=\overline{1,k} \text{; and }\\& uL(\oplus_{n\geq k+1} \G_n)^{t_{k+1}}u^*=L(\oplus_{n\geq k+1} \La_n \oplus A).\end{split}
\end{equation}  
\end{theorem}

\begin{proof} Using Theorem \ref{TwoFoldedProduct} there exist a product decomposition $\La=\La_1 \oplus \Theta_1$, $v_1\in \mathcal U(M)$, and $t_1>0$ such that   $v_1L(\G_1)^{t_1}v_1^* =L(\La_1)$ and $v_1L(\G_{\mathbb N\setminus \{1 \}})^{t/t_1}v^*_1=L(\Theta_1)$. Applying Theorem \ref{TwoFoldedProduct} again in the last relation for the group $\G_{\mathbb N\setminus \{1\} }$ there exist a product decomposition $\Theta_1=\Lambda_2\oplus \Theta_2$, $v_2\in\mathcal U(v_1L(\G_{\mathbb N\setminus \{1 \}})^{t/t_1}v^*_1)$ , and $t_2>0$ such that  $v_2L(\G_2)^{t_2}v_2^* =L(\La_2)$ and $v_2L(\G_{\mathbb N\setminus \{1,2 \}})^{t/(t_1t_2)}v^*_2=L(\Theta_2)$.  Proceeding inductively one has $\Theta_{n-1}=\Lambda_n\oplus \Theta_n$, a unitary $v_n\in \mathcal U(v_{n-1}L(\G_{\mathbb N\setminus \overline{1,n-1} })^{t/(t_1t_2\cdots t_{n-1})}v^*_{n-1})$ and $t_n>0$ such that  $v_nL(\G_n)^{t_n}v_n^* =L(\La_n)$ and $v_nL(\G_{\mathbb N\setminus \overline{1,n} })^{t/(t_1t_2\cdots t_n)}v^*_n=L(\Theta_n)$. Altogether these relations show that $\Theta_n\geqslant \Theta_{n+1}$ for all $n$ and also  $\La=\oplus_\mathbb N\La_n \oplus \Sigma$, where $A =\cap_n \Theta_n$. In addition, for every $k\in \mathbb N$ letting $u_k:=v_1v_2\cdots v_k$ we see that  \begin{equation}\label{intertwiningtheparts}\begin{split}
&u_k L(\G_i)^{t_i}u^*_k=L(\La_i)\text{ for all }i=\overline{1,k} \text{ and}\\ & u_kL(\G_{\mathbb N\setminus \overline{1,k}})^{t/(t_1t_2\cdots t_k)} u^*_k=L(\oplus_{i\geq k+1} \La_i \oplus A).
\end{split} 
\end{equation}  Since $L(\G_k)$ is a II$_1$ factor the second relation in \eqref{intertwiningtheparts} show that for each $k\in \mathbb N$ one can find $u_k\in \mathcal U(M)$ such that  $u_k^*L(A)u_k\subseteq L(\G_{\mathbb N \setminus \overline{1,k}})^{t/(t_1t_2\cdots t_k)}$. Using Proposition \ref{relamenable} and the same argument as in the proof of Claim \ref{relamen} if follows that $A$ is icc amenable as desired. \end{proof}

\begin{Remarks} We conjecture that  Theorem \ref{infprodrig} still holds true without the property (T) assumption on the $\G_i$'s. We point out that property (T) was used in the proof of Theorem \ref{TwoFoldedProduct} only to derive relation  \eqref{eq:CommutantIntertwinesFiniteOrbit}; in other words the (increasing) sequence of subgroups $\Omega_k$ becomes stationary. We believe this conclusion can still be achieved without the property (T) assumption.  However at this time we are unable to prove this.
\end{Remarks}

\noindent \emph{Proof of Corollary \ref{maincor1}}. First we argue that the group $\G=\oplus_n \G_n$ has trivial amenable radical. So let $B\lhd \G$ be a normal amenable subgroup. Thus the von Neumann subalgebra $L(B)\subseteq L(\G)= L(\oplus_{n\neq k}\G_k)\bar\otimes L(\G_k)$ is regular and amenable. Applying \cite[Theorem 1.4]{PV12} it follows that $L(B)\prec L(\oplus_{n\neq k}\G_n)$. Since $B$ is normal $\G$ we further deduce from \cite[Lemma 2.2]{CI17} that $[B: B_k]<\infty$ where $B_k:= B\cap (\oplus_{n\neq k}\G_n)<\oplus_{n\neq k}\G_n$. Since $B_k\lhd\G$ is normal it follows that $B/B_k \lhd \G/B_k$ is a finite normal subgroup. As $\G/B_k= (\oplus_{n\neq k}\G_n/ B_k) \oplus \G_k$, if $\pi_k: \G/B_k\ra \G_k$ is the canonical projection map  it follows that $\pi_k(B/B_k)\lhd \G_k$ is a finite normal subgroup. As $\G_k$ is icc we have  $\pi_k(B/B_k)=1$ and hence $B/B_k <  \oplus_{n\neq k}\G_n/B_k$; in particular, $B=B_k<  \oplus_{n\neq k}\G_n$. Since this holds for every positive integer $k$ then $B <\cap_k  (\oplus_{n\neq k}\G_n)=1$, thus giving the desired claim.

\cite[Theorem 1.3]{BKKO14} implies that  the reduced $C^*$-algebra $\mathbb C^*_r(\G)$ has the unique trace property. Letting $\phi: \mathbb C^*_r(\G)\ra  \mathbb C^*_r(\Lambda)$ be a $\ast$-isomorphism of $C^*$-algebras it follows that $\phi$ lifts to a $\ast$-isomorphism  $\phi: L(\G)\ra  L(\Lambda)$ of von Neumann algebras. By Theorem \ref{infprodrig} we have that $\La= (\oplus_n \La_n)\oplus A$ with $A$ icc amenable; moreover, the corresponding relations (\ref{intertwiningtheparts'}) also hold. Since $\mathbb C^*_r(\Lambda)$ has the unique trace property then \cite[Theorem 1.3]{BKKO14} implies that $A=1$ and the first part of the conclusion is proved. The remaining part of the conclusion follows directly from relations \ref{intertwiningtheparts'}.  $\hfill\square$


\section{Proof of Theorem \ref{main3}}




To ease our exposition we first introduce the following notation:  
\begin{notation}\label{wreathprod} 
 Let $H_0, \G$ be icc groups such that $H_0$ has property (T) and $\G$ admits an infinite, almost normal subgroup $\G_0\leqslant\G$ with relative property (T). Let $\Gamma \curvearrowright I$ be an action on a countable set $I$ satisfying the following conditions:
\begin{enumerate}
\item [a)]  For each $i \in I$ we have $[\G:{\rm Stab}_\G(i)]<\infty$;
\item [b)] There is $k\in \mathbb N$ such that for each $J\subseteq I$ with  $|J|\geq k$ we have $|{\rm Stab}_\G(J)|<\infty$.
\end{enumerate}
Denote by $G=H_0 \wr_I \G$ the corresponding generalized wreath product.
Denote by $M=L(G)$ and assume that  $M=L(\Lambda) $ for an arbitrary group $\La$.  Let $\Delta: M\to M\bar\otimes M $ be the commultiplication along $ \Lambda$, i.e.\ $\Delta(v_\lambda)=v_\lambda\otimes v_\lambda$, where $\{v_\lambda\}_{\lambda\in \Lambda} $ are the canonical unitaries generating $L(\Lambda)$.  
\end{notation}

\begin{prop}\label{semidirprodrig1} Assume Notation \ref{wreathprod}. Then the following hold: \begin{enumerate}
\item [c)]\label{1401}$\Delta(L(H_0^{(I)})) \prec_{M\otimes M}^s L(H_0^{(I)})\bar{\otimes} L(H_0^{(I)})$;
\item [d)]There exists $ u\in \mathcal U(M\bar\otimes M)$ such that $u \Delta(L(\G) )u^*\subseteq L(\G \times \G).$  
\end{enumerate}
\end{prop}

\begin{proof}
We denote $A_0=L(H_0)$ and $A=A_0^{(I)}$. Note that $M=A\rtimes\G$, the action being given by generalized Bernoulli shifts. Write $M=L(\La)$ and denote by $\Delta:M\to M\bar{\otimes} M$ the associated co-multiplication. Note that $M\bar{\otimes} M=(A \bar{\otimes} A)\rtimes (\G\times\G)$.

The inclusion $\Delta(A_0) \subset M\bar{\otimes}M=M\bar{\otimes}(A\rtimes \G)$ is rigid. Denote by $P\subset M\bar{\otimes}M$ the quasi-normalizer of $\Delta(A_0)$. Note that $\Delta(A)\subset P$. By applying \cite[Theorem 4.2]{IPV10}, we see that one of the following has to hold:
\begin{enumerate}
\item $\Delta(A_0)\prec_{M\bar\otimes M} M\otimes 1$;
\item $P\prec_{M\bar\otimes M} M\bar{\otimes}(A \rtimes{\rm Stab}_\G(i))$, for some $i\in I$;
\item $v^*Pv\subset M\bar{\otimes}L(\G)$ for some partial isometry $0\neq v \in M\bar{\otimes}M$.
\end{enumerate}
(1) is impossible since $A_0$ is diffuse. Suppose (3) holds. Then by the remark above we have that $v^*\Delta(A)v\subset M\bar{\otimes}L(\G)$. There are two possibilities: either $\Delta(A)\prec_{M\bar \otimes M} M\bar{\otimes}L({\rm Stab}_\G(i))$ for some $i$, or $\Delta(A)\nprec_{M\bar \otimes M} M\bar{\otimes}L({\rm Stab}_\G)$, for all $i\in I$. In the first case we again have two possibilities: either there exists a maximal finite subset $\mathcal G \ni i$ such that $\Delta(A)\prec_{M\bar \otimes M} M\bar{\otimes}L({\rm Stab}_\G(\mathcal G))$, or there is no such subset. If the first sub-case holds then \cite[Lemma 4.1.3]{IPV10} gives that $QN_{M\bar \otimes M}(\Delta(A))''\prec_{M\bar \otimes M} M\bar{\otimes}L({\rm Norm}(\mathcal G))$. Since the quasi-normalizer of $\Delta(A)$ contains $\Delta(M)$, this implies $\Delta(M)\prec_{M\bar \otimes M} M\bar{\otimes}L({\rm Norm}(\mathcal G))$. As ${\rm Stab}_\G(\mathcal G)$ is a finite index subgroup of ${\rm Norm}(\mathcal G)$, it follows that $\Delta(M)\prec_{M\bar \otimes M} M\bar{\otimes}L({\rm Stab}_\G(\mathcal G))$, and hence $\Delta(M)\prec_{M\bar \otimes M} M\bar{\otimes}L({\rm Stab}_\G(i))$ for some $i$, which implies by \cite[Lemma 7.2.2]{IPV10} that $L({\rm Stab}_\G(i))\subset M$ has finite index, which is a contradiction. If the second sub-case holds, by taking $\mathcal G$ with $|\mathcal G|\geq \kappa$ we get that $\Delta(A)\prec_{M\bar\otimes M} M\otimes 1$, a contradiction. It follows that (2) must hold, hence $P\prec_{M\bar\otimes M} M\bar{\otimes}(A \rtimes {\rm Stab}_\G(i))$, which further implies $\Delta(A)\prec_{M\bar\otimes M} M\bar{\otimes}(A \rtimes {\rm Stab}_\G(i))$, for some $i\in I$. Again we have two possibilities: either there exists a finite maximal subset $\mathcal G\subset I$ such that $\Delta(A)\prec_{M\bar\otimes M} M\bar{\otimes}(A\rtimes {\rm Stab}_\G(\mathcal G))$, or it doesn't. In the first sub-case we get, by \cite[Lemma 4.1.3]{IPV10}, that $QN_{M\bar \otimes M}(\Delta(A))''\prec_{M\bar \otimes M} M\bar{\otimes}(A\rtimes Norm(\mathcal G))$ and again as above, that $\Delta(M)\prec_{M\bar \otimes M} M\bar{\otimes}(A\rtimes {\rm Stab}_\G(i))$, for some $i\in \mathcal G$, which by \cite[Lemma 7.2.2]{IPV10} implies that $[M:A\rtimes{\rm Stab}_\G(i)]$ is finite, a contradiction. In the second sub-case, by taking a $\mathcal G$ with $|\mathcal G|\geq \kappa$, we obtain that $\Delta(A)\prec_{M\bar \otimes M} M\bar{\otimes} A$, which is what we wanted. The maximal projection $q\in \Delta(A)'\cap M\bar{\otimes}M$ such that $\Delta(A)q \prec^s_{M\bar \otimes M} M\bar{\otimes}A$ is non-zero and belongs to the center of the normalizer of $\Delta(A)$ in $M\bar{\otimes}M$. This center is contained in $\Delta(M)'\cap M\bar{\otimes}M=\C 1$. It follows that $q=1$, hence $\Delta(A)\prec^s_{M\bar \otimes M} M\bar{\otimes}A$. By symmetry we obtain that also $\Delta(A)\prec^s_{M\bar \otimes M} A\bar{\otimes}M$ and finally that $\Delta(A)\prec^s_{M\bar \otimes M} A \bar{\otimes}A$, showing part c).
\vskip 0.04in
Next we prove part d). First notice from the assumptions that the inclusion $\Delta(L(\G_0))\subset M\bar{\otimes}(A\rtimes\G)$ is rigid. Denote by $P$ the quasi-normalizer of $\Delta(L(\G_0))$ inside $M\bar{\otimes}M$. Note that $P$ contains $\Delta(L(\G))$. We apply again \cite[Theorem 4.2]{IPV10} and we see that one of the following has to hold:
\begin{enumerate}
\item $\Delta(L(\G_0))\prec_{M\bar\otimes M} M\bar{\otimes}1$;
\item $P \prec_{M\bar\otimes M} M\bar{\otimes}(A\rtimes {\rm Stab}_\G(i))$, for some $i\in I$;
\item $vPv^*\subset M\bar{\otimes}L(\G)$, for some $v\in \mathcal U(M\bar{\otimes}M)$.
\end{enumerate}
Note (1) cannot be true because $\Delta(L(\G_0))$ is diffuse. Suppose (2) is true. This implies in particular that $\Delta(L(\G)) \prec_{M\bar\otimes M} M\bar{\otimes}(A \rtimes {\rm Stab}_\G(i))$. But since $\Delta(A)\prec^s_{M\bar\otimes M} A\bar{\otimes}A$, by the same argument as in the beginning of the proof of \cite[Theorem 8.2]{Io10}, we would get $\Delta(A\rtimes\G)=\Delta(M)\prec_{M\bar\otimes M} M\bar{\otimes}(A\rtimes {\rm Stab}_\G(i))$, which by  \cite[Lemma 7.2.2]{IPV10} implies that $A\rtimes {\rm Stab}_\G(i) \subset M$ has finite index, a contradiction. So (3) must be true, hence a fortiori $v\Delta(L(\G))v^*\subset M\bar{\otimes}L(\G)$. Repeating the argument for the inclusion $v\Delta(L(\G))v^*\subset M\bar{\otimes}L(\G)=(A\rtimes\G)\bar{\otimes}L(\G)$, we obtain an unitary $u\in M\bar{\otimes}M$ such that $u\Delta(L(\G))u^*\subset L(\G)\bar{\otimes}L(\G)$, as desired.\end{proof}

\begin{prop}\label{semidirprodrig1'} Assume Notation \ref{wreathprod}. In addition assume that $H_0$, $\G$ and $\La$ are torsion free groups. Then the following hold: \begin{enumerate}
\item [e)]\label{1401}$\Delta(L(H_0^{(I)})) \prec_{M\otimes M}^s L(H_0^{(I)})\bar{\otimes} L(H_0^{(I)})$;
\item [f)]There exists $ w\in \mathcal U(M\bar\otimes M)$ such that $w \Delta(\G)w^*\subseteq \mathbb T(\G \times \G).$  
\end{enumerate}
\end{prop}

\begin{proof} Since e) follows directly from Proposition \ref{semidirprodrig1} we will only argue for f). From Proposition \ref{semidirprodrig1} there exists $ u\in \mathcal U(M\bar\otimes M)$ such that $u\Delta(L(\G) )u^*\subseteq L(\G \times \G)$.   

Denote by $\mathcal G=\{u\Delta(u_\g)u^*\,|\, \g\in\G\}\subset\mathcal U(L(\G\times\G))$. Since $\mathcal G$ normalizes $u\Delta(A)u^*$, which satisfies $u\Delta(A)u^* \prec^s A\bar{\otimes}A$, the argument in Step 5 of the proof of  \cite[Theorem 5.1]{IPV10} implies that $h_{\G\times\G}(\mathcal G)>0$. We have that $\mathcal G''=u\Delta(L(\G))u^*\nprec_{M\bar\otimes M} L(C_{\G\times\G}(\g_1,\g_2))$, for any $(\g_1,\g_2)\in\G\times\G-\{e\}$. Indeed, suppose this is not true, and $u\Delta(L(\G))u^* \prec L(C_{\G\times\G}((\g_1,\g_2)))=L(C_{\G}(\g_1))\bar{\otimes}L(C_{\G}(\g_2))$, with $\g_2\neq e$. Again using the fact that $\Delta(A)\prec_{M\bar\otimes M} A\bar{\otimes}A$, we infer that $\Delta(M)=\Delta(A\rtimes\G)\prec M\bar{\otimes}(A\rtimes C_{\G}(\g_2))$. \cite[Lemma 7.2.2]{IPV10} then implies that $A\rtimes C_{\G}(\g_2)$ has finite index in $M$, which is a contradiction, since $\G$ is icc. Also, the representation $\{Ad (v)\}_{v\in\mathcal G}$ on $L^2(L(\G\times\G))\ominus\C 1$ is weakly mixing, because it is in fact weakly mixing on $L^2(M\bar{\otimes}M)\ominus \C1$. Indeed, let $\mathcal H\subset L^2(M\bar{\otimes}M)$ be a finite dimensional $\{Ad (v)\}_{v\in\mathcal G}$ invariant subspace. Then $\mathcal H_0=u\mathcal Hu^*$ is a finite dimensional $\{Ad \Delta(u_\g)\}_{\g\in\G}$-invariant subspace of $L^2(M\bar{\otimes}M)$. Denote by $\mathcal K$ the closed linear span of $\mathcal H_0\Delta(M)$. Then $\mathcal K$ is a $\Delta(L\G)-\Delta(M)$ bi-module, which is finitely generated as a right module. Since $L(\G)\nprec_M L(C_{\La}(s))$, for any $s\in\La-\{e\}$, \cite[Proposition 7.2.3]{IPV10} implies that $\mathcal K\subset \Delta(L^2M)$, so in particular $\mathcal H_0 \subset \Delta(L^2M)$. Hence $\Delta^{-1}(\mathcal H_0)\subset L^2M$ is a finite dimensional $\{Ad (u_\g)\}_{\g\in\G}$-invariant subspace. As the inclusion $\G \leqslant H_0^{(I)}\rtimes\G$ is icc, the representation $\{Ad (u_\g)\}_{\g\in\G}$ on $L^2(M)\ominus \C1$ is weakly mixing, which further implies that $\mathcal H=\C1$, as claimed. Now we apply \cite[Theorem 4.1]{KV15} to deduce that there exists a unitary $w\in L(\G\times\G)$ such that $w\mathcal G w^*\subset \mathbb T(\G\times\G)$. By replacing $w$ with $w u$, we may assume that $w\Delta(u_\g)w^*\in\mathbb T(\G\times\G)$ for all $\g\in\G$.\end{proof}

\begin{theorem}\label{strongrigsemidirect1} Let $H_0, \G$ be icc torsion free groups such that $H_0$ has property (T) and $\G$ admits an infinite, almost normal subgroup $\G_0\leqslant\G$ with relative property (T). Let $\Gamma \curvearrowright I$ be a transitive action on a countable set $I$ satisfying the following conditions:
\begin{enumerate}
\item [a)]  For each $i \in I$ we have $[\G:{\rm Stab}_\G(i)]<\infty$;
\item [b)] There is $k\in \mathbb N$ such that for each $J\subseteq I$ satisfying  $|J|\geq k$ we have $|{\rm Stab}_\G(J)|<\infty$;
\item [c)] For every $i\neq j$ we have that $|{\rm Stab}_\G(i)\cdot j|=\infty$. 
\end{enumerate}
Denote by $G=H_0 \wr_I \G$ the corresponding generalized wreath product. Let $\La$ be any torsion free group and let $\theta: L(G)\ra L(\La)$ be a $\ast$-isomorphism. Then $\La$ admits a wreath product decomposition $\La = \Sg_0 \wr_I \Psi$ satisfying the following properties: there exist a group isomorphism $\rho:\G \ra \Psi$, a character $\eta:\G\ra \mathbb T$, a $\ast$-isomorphism $\theta_0: L(H_0)\ra L(\Sg_0)$ and a unitary $v\in L(\La)$ such that for every $x\in L(H_0^{(I)})$ and $\g\in \G$ we have \begin{equation*}
\theta(x u_\g) =\eta(\g)v^* \theta^{\bar\otimes I}_0(x) v_{\delta(\g)}v.\end{equation*} 
Here $\{u_\g \,|\, \g\in \G\}$ and $\{v_\la \,|\, \la\in \Psi\}$ are the canonical unitaries of $L(\G)$ and $L(\Psi)$, respectively. \end{theorem}
\begin{proof} Let $A=L(H_0^{(I)})$, and notice from assumptions we have that $\theta(L(G))=L(\La)=M$. Using Proposition \ref{semidirprodrig1'} one can find  $w\in \mathcal U(M\bar\otimes M)$, group homomorphisms $\delta_i:\G\to\G$, and a character $\omega:\G \to \mathbb T$ such that $w\Delta(\theta(u_\g))w^*=\omega(\g)\theta(u_{\delta_1(\g)})\otimes \theta(u_{\delta_2(\g)})$ for all $\g\in\G$. Then aplying verbatim Steps 4 and 5 in the proof of \cite[Theorem 8.2]{IPV10} one can find an injective group homomorphism $\rho: \G \ra \La$ and a character $\eta: \G \ra \mathbb T$ satisfying \begin{equation}\label{eqgroups}\theta(u_\g)= \eta(\g)v_{\rho(\g)},\text{ for all }\g\in \G.\end{equation} Denote by $\Psi=\rho(\G)$. In addition, these proofs also show there is $v\in \mathcal U(M)$ such that $w=(v^*\otimes v^*)\Delta(v)$. Henceforth the canonical unitaries $\theta(u_\g), \g \in \G$ will be replaced by $v\theta(u_\g)v^*$ and $A$ will be replaced by $vAv^*$.
Under these conventions we prove that
\begin{claim}\label{contcore}$\Delta(\theta(A))\subset \theta(A)\bar{\otimes}\theta(A)$.
\end{claim}
\noindent \emph{Proof of Claim \ref{contcore}}.  By Proposition \ref{semidirprodrig1}, $\Delta(\theta(A))\prec^s \theta(A)\bar{\otimes}\theta(A)$. This means that for every $\epsilon>0$, there exists a finite subset $e\in S\subset \theta(\G)$ such that $\|d-P_{S\times S}(d)\|_2\leq\epsilon$, for all $d\in \mathcal U(\Delta(\theta(A)))$. But since, according to Proposition \ref{semidirprodrig1'}, $\Delta(\theta(A))$ is invariant to  $Ad (\Delta(\theta(u_\g)))= Ad (\theta(u_\g)\otimes \theta(u_\g))$ for all $\g\in\G$, we see that $\|d-P_{\mu S \mu^{-1}\times \mu S\mu^{-1}}(d)\|_2\leq \epsilon$, for all $d\in\mathcal U(\Delta(\theta(A)))$ and $\mu\in \theta(\G)$. As $\G$ is icc we can find $\mu\in\theta(\G)$ such that $\mu S \mu^{-1}\cap S=\{e\}$ (see for instance \cite[Proposition 3.4]{CSU13}). By the triangle inequality this further implies that
\[\|d-E_{\theta(A)\bar{\otimes}\theta(A)}(d)\|_2=\|d-P_{(\mu S \mu^{-1}\cap S)\times (\mu S\mu^{-1}\cap S)}(d)\|_2\leq 2\epsilon,\]
for all $d\in\mathcal U(\Delta(\theta(A)))$. As $\epsilon$ is arbitrary, this implies $\Delta(\theta(A))\subset \theta(A)\bar{\otimes}\theta(A)$. $\hfill\blacksquare$
\vskip 0.04in
From Claim \ref{contcore} and \cite[Lemma 7.1.2]{IPV10} it follows that $\theta(A)=L(\Sigma)$, for some $\Sigma<\La$. Since the $u_\g$'s normalize $A$, it follows that $\Psi\ni \rho(\g)$ normalizes $\Sigma$, for all $\g$. Consider the action of $\Psi \ra Aut(\Sigma)$ given by $\Psi\ni \la \ra Ad(\la)\in Aut (\Sg)$ and observe $\La$ splits as a semidirect product  $\La=\Sigma\rtimes\Psi$, because $L(\La)=\theta (A)\rtimes\theta(\G)$. 

For the remaining part consider  $A_0=L(H_0)$ and denote by $A_0^i$ the copy of $A_0$ in position $i\in I$.  Next we show that 
\begin{claim}\label{contcoreslot}$\Delta(\theta(A^i_0))\subset \theta(A^i_0)\bar{\otimes}\theta(A^i_0)$, for all $i \in I$.
\end{claim}
\noindent \emph{Proof of Claim \ref{contcoreslot}}. Using \eqref{eqgroups} we note that $\Delta(\theta(A_0^i))$ is fixed by $Ad(\Delta(\theta (u_\g)))=Ad(\theta(u_\g)\otimes \theta(u_\g))$, for all $\g\in {\rm Stab}_\G (i)$. Due to the assumption that ${\rm Stab}_\G(i)\cdot j$ is infinite for all $i\neq j$, the representation  $Ad \{\theta(u_\g)\otimes \theta(u_\g)\}_{\g\in {\rm Stab}_\G (i)}$ is weakly mixing on $L^2(M\bar{\otimes}M)\ominus L^2(\theta(A_0^i) \bar\otimes \theta(A_0^i))$, so it follows that $\Delta(\theta(A_0^i))\subset \theta(A_0^i) \bar{\otimes} \theta(A_0^i)$. $\hfill\blacksquare$ 

Hence from Claim \ref{contcoreslot} and \cite[Lemma 7.1.2]{IPV10}  for every $i\in I$ there exists a subgroup $\Sigma_i<\La$ such that $\theta(A_0^i)=L(\Sigma_i)$. Since the action $\G\curvearrowright I$ is transitive, it follows that $\Sigma_i\cong \Sigma_0$ for all $i$, and then that $\Sigma=\bigoplus_I \Sigma_0$. Moreover, this entails that the action $\Psi\ra Aut(\Sg)=Aut(\bigoplus_I \Sg_i)$ is induced by the generalized Bernoulli action of $\Psi\ca I$ and hence $\Lambda=\Sigma_0 \wr_I \G$. The rest of the statement follows from the previous observations. \end{proof}


\section{Proof of Theorem \ref{main4}}

\begin{theorem}\label{semidirprodrig2}Let $H_0, \G$  be icc, property (T) group. Also assume that $\G=\G_1\times \G_2$, where $\G_i$ are nonamenable biexact groups for all $i=1,2$. Let $\Gamma \curvearrowright I$ be an action on a countable set $I$ satisfying the following conditions:
\begin{enumerate}
\item [a)] The stabilizer ${\rm Stab}_\G(i)$ is amenable for each $i \in I$;
\item [b)] There is $k\in \mathbb N$ such that for each $J\subseteq I$ satisfying  $|J|\geq k$ we have $|{\rm Stab}_\G(J)|<\infty$.
\end{enumerate}
Denote by $G=H_0 \wr_I \G$ the corresponding generalized wreath product. Let $\La$ be an \emph{arbitrary} group and let $\theta: L(G)\ra L(\La)$ be a $\ast$-isomorphism. Then $\La$ admits a semidirect product decomposition $\La = \Sg \rtimes \Phi$ satisfying the following properties: there exist a group isomorphism $\delta:\G \ra \Phi$, a character $\zeta:\G\ra \mathbb T$, a $\ast$-isomorphism $\theta_0: L(H_0^{(I)})\ra L(\Sg)$ and a unitary $t\in L(\La)$ such that for every $x\in L(H_0^{(I)})$ and $\g\in \G$ we have \begin{equation*}
\theta(x u_\g) =\zeta(\g)t \theta_0(x) v_{\delta(\g)}t^*.
\end{equation*}
Here $\{u_g \,|\, \g\in \G\}$ and $\{v_\la \,|\, \la\in \Phi\}$ are the canonical unitaries of $L(\G)$ and $L(\Phi)$, respectively.

\end{theorem}
\begin{proof} From assumptions we have that $\theta(L(G))=L(\La)=M$. Denote by $A_0=\theta(L(H_0))$ and $A=\theta(L(H_0^{I})$. Also to simplify the writing, throughout the proof we will identify $\G$ with $\theta(\G)$, etc. Thus note that $M=A\rtimes\G$, the action being given by generalized Bernoulli shifts. Consider $\Delta:M\to M\bar{\otimes} M$ the comultiplication along $\Lambda$. Note that $M\bar{\otimes} M=(A \bar{\otimes} A)\rtimes (\G\times\G)$. Theorem \ref{semidirprodrig1} implies that 
\begin{enumerate}
\item \label{1206} $\Delta(A)\prec^s_{M\bar\otimes M} A\bar\otimes A$, and
\item \label{1207} there is $u\in \mathcal U(M\bar\otimes M)$ such that $u\Delta(L(\G))u^*\subseteq L(\G\times \G)$.
\end{enumerate}

\noindent Next we show the following 
\begin{claim}\label{groupidentificationcorner} There exist a subgroup $\Phi< \La$ with $QN^{(1)}_\La(\Phi)=\Phi$, $d\in \mathcal P(L(\Phi))$ and  $\mu\in \mathcal U(M)$ satisfying $h=\mu d\mu^*\in L(\G)$ and $\mu dL(\Phi)d\mu^*= hL(\G)h$. \end{claim}
\noindent\emph{Proof of Claim \ref{groupidentification}}. Let $\mathcal K:=\{ \G\times \G_1, \G\times \G_2, \G_1 \times \G,\G_2\times \G  \}$. Since by \cite[Lemma 15.3.3]{BO08} $\G\times \G$ is biexact relatively to $\mathcal K$ and  $\Delta(L(\G_1))$ and $\Delta(L(\G_2))$ are commuting non-amenable factors then \cite[Theorem 15.1.5]{BO08} implies that there are $\Psi\in \mathcal K$ and $i=1,2$ so that $u\Delta(L(\G_i))u^*\prec_{L(\G\times \G)} L(\Psi)$. Since the flip automorphism of $M\bar\otimes M$ acts identically on $\Delta (L(\G_i))$ we can assume without any loss of generality $\Psi=\G\times \G_1$ and $i=1$. Hence \begin{equation*}
u\Delta(L(\G_1))u^*\prec_{L(\G\times \G)} L(\G\times \G_1).
\end{equation*}    
The using \cite[Theorem 4.1]{DHI16} (see also \cite[Theorem 3.3]{CdSS16})) this further implies there exists a subgroup $\Sg<\La$ with  non-amenable centralizer $\Upsilon:=C_\La(\Sg)$  and $L(\G_1)\prec_M L(\Sg)$. Passing to the intertwining of the relative commutants we have that  $L(\Upsilon)\subseteq L(\Sg)'\cap M\prec_M L(\G_1)'\cap M=L(\G_2) $. Thus there are projections $e\in L(\Upsilon)$,$f\in L(\G_2)$, a partial isometry $v\in M$, and an injective unital $\ast$-homomorphism  $\phi: eL(\Upsilon)e\ra fL(\G_2)f$ such that \begin{equation}\label{1200}
\phi(x)v=vx, \text{ for all } x\in eL(\Upsilon)e.
\end{equation} 
Denote by $T:= \phi(eL(\Upsilon)e)$ and notice that $q':=vv^*\in T'\cap fMf $ and $p:=v^*v\in eL(\Upsilon)e'\cap eMe = (L(\Upsilon)'\cap M) e$.  Since $T$ is a non-amenable factor then $a)$ implies that $T\nprec L({\rm Stab}_\G (i))$ for all $i$ and using \cite[Theorem 3.1]{Po03} we have $QN_{fMf} (T)''\subseteq L(\G)$. In particular, $q'\in L(\G)$ and by \eqref{1200} there is $u\in \mathcal U(M\bar\otimes M)$ such that $u eL(\Upsilon)epu^*\subseteq L(\G)$. Since $L(\G)$ is a factor, the same argument from \cite[Theorem 5.1, page 26]{IPP05} shows that one can perturb $u$ to a new unitary such that we further have  \begin{equation}\label{1201}u L(\Upsilon)pu^*\subseteq L(\G).
\end{equation}  Since $\Upsilon$ is non-amenable then $uL(\Upsilon)pu^*\nprec L({\rm Stab}_\G (i))$ for all $i$ and \eqref{1201} combined with  \cite[Theorem 3.1]{Po03} and the quasinormalizer formula show that $ u pL(QN_\La(\Upsilon))''pu^*\subseteq  QN_{upM pu^*}(uL(\Upsilon)pu^*)''\subseteq L(\G)$. Since $L(\G)$ is a factor, the same argument as before further implies that $u L(QN_\La(\Upsilon))z' u^*\subseteq L(\G)$, where $z'$ is the central support of $p$ in $L(QN_\La(\Upsilon))$. Notice $\Sg\leqslant vC_\La(\Upsilon)< QN_\La(\Upsilon)$ and hence $u L(\Sg)z' u^*\subseteq L(\G)$. Letting $\Omega:=vC_\La(\Sg)$, $\Theta
:=QN^{(1)}_\La(\Sg\Omega)$ and using the same arguments as before, we can further find $\eta\in \mathcal U (M\bar\otimes M)$ and a projection $z\in \mathcal Z(L(\Theta))$ such that \begin{equation}\label{1202}\eta L(\Theta) z \eta^*\subseteq L(\G).
\end{equation}   Since  $\Upsilon,\Sg<\Theta$ are commuting non-amenable groups and  $\G$ is biexact relatively to $\{\G_1,\G_2\}$, \cite[Theorem 15.1.5]{BO08} implies that $\eta L(\Sg) z\eta^*\prec_{L(\G)} L(\G_k)$, for some $k=1,2$. Again, wlog we can assume $k=1$. Passing to the relative commutants intertwining we get \begin{equation}\label{1203}L(\G_2)=L(\G_1)'\cap L(\G)\prec_{L(\G)}  (\eta L(\Sg) z\eta^*)'\cap \eta z\eta^* L(\G) \eta z\eta ^*\subseteq \eta L(\Omega)z\eta^*.\end{equation} 

Now let $\{\mathcal O_k\}_k$ be a countable enumeration of all the finite orbits under conjugation by $\Sg$ and notice that $\cup_k\mathcal O_k=\Omega$. Consider $\Omega_k:=\langle\mathcal O_1,...,\mathcal O_k\rangle\leqslant \La$ and note that $\Omega_k\nearrow \Omega$. Since $L(\G_2)$ has property (T) then \eqref{1203} implies that  \begin{equation}L(\G_2)\prec_{L(\G)}  \eta L(\Omega_k)z\eta^*\text{ for some }k.
\end{equation} 
By \cite[Proposition 2.4]{CKP14}, there exist nonzero projections $a\in L(\G_2), q\in L(\Omega_k) $, a partial isometry $w\in L(\G)$, a subalgebra $D\subseteq \eta qL(\Omega_k)qz\eta^*$, and a $\ast $-isomorphism $\psi:a L(\G_2)a\to D  $ such that 
\begin{align}
D\vee (D'\cap \eta qL(\Omega_k)qz\eta^*)\subseteq \eta qL(\Omega_k)qz\eta^*\quad \text{has finite index, and}\label{eq:DHasFiniteIndex'}\\
\psi(x)w=wx \quad \forall \, x\in aL(\G_2)a.\label{eq:PartialIsomIntertwinesA'}
\end{align}
Let $r=\eta qz\eta^* $ and notice that $ww^*\in D'\cap r L(\G) r $ and $w^*w\in (aL(\G_2)a)'\cap aL(\G)a=L(\G_1)\bar\otimes \mathbb C a $.  Hence there is a projection $b\in L(\G_1)$ satisfying $w^*w=b\otimes a $.  Picking $c\in\mathcal{U}(L(\G)) $ so that $w=c (b\otimes a) $ then \eqref{eq:PartialIsomIntertwinesA'} gives   
\begin{align}\label{cornerbegin'}
Dww^*=w L(\G_2)w^*=c(\mathbb Cb\otimes aL(\G_2)a)c^*.
\end{align}
Passing to the relative commutants, we obtain $ww^*(D'\cap r L(\G) r) ww^*=c(bL(\G_1)b \otimes \mathbb C a) c^*$. Hence there exist $s_1,s_2> 0 $ satisfying \begin{align}\label{corner'}(D'\cap r L(\G) r)y=c(bL(\G_1)b \otimes \mathbb Ca)^{s_2}c^*\cong L(\G_1)^{s_1},
\end{align}
where $y $ is the central support projection of $ ww^*$ in $D'\cap r L(\G) r$.   Notice 
\begin{align*}
D'\cap rL(\G)r\supseteq (\eta qL(\Omega_k)qz\eta^*)'\cap rL(\G)r= \eta (L(\Omega_k)'\cap L(\Theta))qz \eta^*\supseteq \eta L(C_\Sg(\Om_k))qz\eta^*.
\end{align*}
From the definition of $\Omega_k$ it follows that  $[\Sg:C_\Sg(\Omega_k)]<\infty $. Since $\Sg$ is non-amenable it follows that $C_\Sg(\Omega_k)$ is also non-amenable and hence $\eta L(C_\Sg(\Omega_k))qz \eta^* $ has no amenable direct summand.  Moreover we also have  $D'\cap rL(\G)r\supseteq D'\cap \eta qL(\Omega_k)qz\eta^* $. In conclusion $(\eta L(C_\Sg(\Omega_k))qz \eta^* )y$ and $(D'\cap \eta qL(\Omega_k)qz\eta^*)y$ are commuting subalgebras of $(D'\cap rL(\G)r)y$ where $(\eta L(C_\Sg(\Omega_k))qz \eta^* )y$ has no amenable direct summand.  Since $\Gamma_i $ was assumed to be bi-exact, then using \eqref{corner'} and \cite[Theorem 1]{Oz03} it follows that $(D'\cap \eta qL(\Omega_k)qz\eta^*)y$ is purely atomic.  Thus, cutting by a central projection $r'\in D'\cap \eta qL(\Omega_k)qz\eta^*$ and using \eqref{eq:DHasFiniteIndex'} we may assume that $D\subseteq \eta qL(\Omega_k)qz\eta^*$ is a finite index inclusion of algebras.  Proceeding as in the second part of \cite[Claim 4.4]{CdSS16}, we may assume that $D\subseteq \eta qL(\Omega_k)qz\eta^* $ is a finite index inclusion of II$_1$ factors. Moreover one can check that if one replaces $w$ by the partial isometry of the polar decomposition of $r'w\neq 0$ then all relations \eqref{eq:PartialIsomIntertwinesA'},\eqref{cornerbegin'} and \eqref{corner'} are still satisfied. 

Using relation \eqref{cornerbegin'}, the quasinormalizer compresion formula, and the fact that $D\subseteq \eta qL(\Omega_k)qz\eta^* $ is a finite index inclusion of II$_1$ factors we can see that  \begin{equation*}  \begin{split} c (b\otimes a)L(\G) (b\otimes a) c^*& = QN_{c (b\otimes a) M (b\otimes a) c^*}(c(\mathbb Cb\otimes (a L(\G_2)a))c^*)'' \\ 
&=QN_{ww^*Mww^*} (Dww^*)'' \\
&= ww^*QN_{rMr}(D)'' ww^*\\ 
&= ww^*QN_{\eta q z M qz \eta^*}(\eta qL(\Omega_k)qz \eta^*)'' ww^*\\ 
&= ww^* \eta qz QN_{L(\La)}(L(\Omega_k))''qz \eta ww^*.\end{split}
\end{equation*}
Letting $\Xi=QN_\La(\Omega_k)$, then the previous relation and formula \cite[Corollary 5.2]{JGS10} imply that   $c (b\otimes a)L(\G) (b\otimes a) c^*=ww^*\eta L(\Xi)\eta^*ww^*$. Since $QN^{(1)}_G(\G) =\G$, this formula and \cite[Corollary 5.2]{JGS10} further imply that  \begin{equation}\label{1205}\begin{split}c (b\otimes a)L(\G) (b\otimes a) c^*& =QN^{(1)}_{c (b\otimes a) M (b\otimes a) c^*}(c (b\otimes a)L(\G) (b\otimes a) c^*)''\\&=QN^{(1)}_{ww^* \eta M\eta^*ww^*}(ww^*\eta L(\Xi)\eta^*ww^*)''\\&= ww^*\eta L(\Phi)\eta^* ww^*,\end{split}\end{equation} where $\Phi = \langle QN^{(1)}_\La(\Xi)\rangle$. Hence in particular we have $ww^*\eta L(\Xi)\eta^* ww^*= ww^*\eta L(\Phi)\eta^* ww^*$ and by \cite[Proposition 2.6]{CdSS16} it follows that $[\Phi:\Xi]<\infty$. This entails that $\Phi=QN^{(1)}_\La(\Xi)=QN^{(1)}_\La(\Phi)$.  Note the above relations imply that $ww^*\in \eta L(\Xi) \eta^* \subseteq \eta L(\Phi)\eta^*$. Consider $d\in \mathcal P(L(\Phi))$ such that $ww^*=\eta d\eta^*$ and letting $\mu:= c^*\eta$ and $h:=b\otimes a$ then relation \eqref{1205} gives the desired conclusion. $\hfill\blacksquare$

\begin{claim}\label{groupidentification} There exists a unitary $w\in M$ such that $wL(\Phi)w^*=L(\G)$. \end{claim}

\noindent\emph{Proof of Claim \ref{groupidentification}}. From Claim \ref{groupidentificationcorner}, there exists $\Phi\leqslant \La$ with $\Phi = QN^{(1)}_\La(\Phi)$, $d\in\mathcal P(L(\Phi))$ and  $\mu\in \mathcal U(M)$ satisfying $h=\mu d\mu^*\in L(\G)$ and \begin{equation}\label{1100}\mu dL(\Phi)d \mu^*=hL(\G)h.
\end{equation} As $\G$ has property (T), \eqref{1100}  implies that $dL(\Phi)d$ is a property (T) von Neumann algebra. By \cite[Lemma 2.13]{CI17} it follows that $\Phi$ is a property (T) group. Fix $r\in \mathcal P((\mu L(\Phi) \mu^*) '\cap M)$ and note that $\mu L(\Phi) \mu^* r$ is a property (T) von Neumann algebra. Thus, using \cite[Theorem]{} we have that either \begin{enumerate}\item $\mu L(\Phi)\mu^* r\prec_M L(\G)$, or \item $\mu L(\Phi) \mu^* r \prec_M L(H^F)$ for some finite $F\subset I$.\end{enumerate} If (2) would hold then we would have that $L(H^{I\setminus F})=L(H^F)'\cap M \prec_M  (\mu L(\Phi) \mu^* r)'\cap rMr =r \mu(L(\Phi)'\cap M)\mu^* r$. On the other hand since  $QN^{(1)}_\La(\Phi)=\Phi$ we have $\mu (L(\Phi)'\cap M)\mu^*= \mathcal Z (\mu L(\Phi)\mu^*)$  . Altogether these would show that $H^{I\setminus F}$ is amenable and hence $H$ is amenable, a contradiction. So (1) must hold for every $r\in \mathcal P( (\mu L(\Phi)\mu^*)'\cap M)$. This entails that $\mu L(\Phi)\mu^* \prec^s_M L(\G)$ and using \eqref{1100} and \cite[Lemma 2.6]{CI17}   one can find $w\in \mathcal U(M)$ such that $wL(\Phi)w^*=L(\G)$.  $\hfill\blacksquare$

Next consider the subgroup $\mathcal G=\{u\Delta(u_\g)u^*\mid \g\in\G\}\leqslant \mathcal U(L(\G\times\G))$. Since $\mathcal G$ normalizes $u\Delta(A)u^*$, which by (\ref{1206}) satisfies $u\Delta(A)u^* \prec^s A\bar{\otimes}A$, the argument in Step 5 in the proof of \cite[Theorem 5.1]{IPV10} implies that $h_{\G\times\G}(\mathcal G)>0$. Then using Lemmas \ref{perturbationheight1}-\ref{perturbationheight2} we further have that $h_{\La}(\G)>0$. Using Lemma \ref{perturbationheight1} we get $h_{\La}(w^*\G w)>0$ and by Claim \ref{groupidentification} we further conclude that $h_{\Phi}(w^*\G w)>0$. Thus by \cite[Theorem 3.1]{IPV10} one can find $t\in \mathcal U(M)$, a character $\zeta:\G\ra \mathbb T$ and a group isomorphism $\delta: \G\ra \Phi$ such that \begin{equation}t u_\g t^* =\overline {\zeta(\g)} v_{\delta (\g)} ,\text{ for all }\g\in \G.
\end{equation} Letting $\Omega:=(t^*\otimes t^*) \Delta(t)\in \mathcal U(M\bar\otimes M)$ we then have  $\Omega\Delta(u_\g)\Omega^*=\zeta(\g)(u_{\g}\otimes u_{\g})$ for all $\g\in\G$. Next, we replace the canonical unitaries $u_\g, \g \in \G$ by $t u_\g t^*$ and $A$ by $tAt^*$. Then (\ref{1207}) combined with the argument from the proof of Claim \ref{contcore} in Theorem \ref{strongrigsemidirect1} further shows that $\Delta(A)\subset A\bar\otimes A$. Hence using \cite[Lemma 7.1.2]{IPV10} there exists a subgroup $\Sg< \Lambda$ such that $A=L(\Sg)$. Since the $u_\g$'s normalize $A$, it follows that $v_{\delta(\g)}$ normalizes $\Sigma$, for all $\g$. Moreover, since $L(\La)=A\rtimes\G$,  then $\La$ admits a semidirect product decomposition $\La=\Sigma\rtimes\Phi$. Altogether the previous considerations give the conclusion of the theorem. \end{proof}



\begin{theorem}\label{wprig2} Let $H$ be icc, weakly amenable, biexact property (T) group. Let $\G=\G_1\times \G_2$, where $\G_i$ are icc, biexact, property (T) group.  Assume that $\G\ca I$ is an action on a countable infinite set $I$ that satisfies the following properties:
\begin{enumerate}
\item [a)] The stabilizer ${\rm Stab}_\G(i)$ is amenable for each $i \in I$;
\item [b)] There is $k\in \mathbb N$ such that for each $J\subseteq I$ satisfying  $|J|\geq k$ we have $|{\rm Stab}_\G(J)|<\infty$.
\end{enumerate}
Let $G=H \wr_I \G$ be the corresponding generalized wreath product. Let $\La$ be an \emph{arbitrary} group and let $\theta: L(G)\ra L(\La)$ be a $\ast$-isomorphism. Then one can find non-amenable icc groups  $\Sg_0,\Psi$, an amenable icc group $A$, and an action $\Psi \ca^\alpha A$ such that we can decompose $\La$ as semidirect product $\La = (\Sg_0^{(I)} \oplus A)\rtimes_{\beta\oplus \alpha} \Psi$, where $\Psi\ca^{\beta} \Sg_0^{(I)}$ is the generalized Bernoulli action. In addition, there exist a group isomorphism $\delta:\G \ra \Psi$, a character $\eta:\G\ra \mathbb T$, a $\ast$-isomorphism $\theta_0: L(H^{(I)})\ra L(\Sg_0^{(I)}\oplus A)$ and $u\in \mathcal U(L(\La))$ so that for every $x\in L(H^{(I)})$ and $\g\in \G$ we have \begin{equation*}
\theta(x u_\g) =\eta(\g)u \theta_0(x) v_{\delta(\g)}u^*.
\end{equation*}
Here $\{u_g \,|\, \g\in \G\}$ and $\{v_\la \,|\, \la\in \Psi\}$ are the canonical unitaries of $L(\G)$ and $L(\Psi)$, respectively. 

\end{theorem}

\begin{proof} Let $G= H\wr_I \G$ satisfies the conditions stated in the Theorem \ref{semidirprodrig2}. Let  $\La $ be an arbitrary group and assume that  $\theta : L(G) \ra L(\La)$ is an $\ast$-isomorphism. Using Theorem \ref{semidirprodrig2}, after composing $\theta$ with an inner automorphism of $M$ one can find a semidirect product decomposition of $\La=\Sg\rtimes_\beta \Psi$, a group isomorphism $\delta : \G\ra \Psi$, and a character $\eta: \G\ra \mathbb T$ such that \begin{equation}\label{402}\theta(u_\g)=\eta(\g)v_{\delta(\g)}\text{ for all }\g\in \G.\end{equation} Moreover, we have $\theta (L(\Sg))=L(H^{(I)})$. Since $H$ are icc, biexact, weakly amenable, property (T) groups then Theorem \ref{main1} implies that one can decompose $\Sg=\oplus_{i\in I} \Sg_i \oplus A$, where $A$ is trivial or amenable icc. In addition, for every finite subset $F\subset I$ there exist $u\in \mathcal U(L(\La))$ and scalars $t_i>0$ for $i\in F$ such that
\begin{equation}\label{401}\begin{split}
&u L(\Sg_i)^{t_i}u^*=\theta(L(H_i)) \text{ for all }i\in F,\text{ and}\\
&uL(\oplus_{i\in I\setminus F } \Sg_i \oplus A)^{\prod_{i\in F} t^{-1}_i}u^*=\theta(L(H_{I\setminus F})).
\end{split}
\end{equation}

Next we show that $\Sg_i\cong \Sg_0$ for all $i$ and there exists an action $\Psi \ca^\alpha A$ such that $\La = (\Sg^{(I)} \oplus A)\rtimes_{b\oplus \alpha} \Psi$, where $\Psi\ca^{b} \Sg^{(I)}$ is the generalized Bernoulli action induced by $\G\ca I$. 
Fix $i,j\in I$ and $\g\in \G$ such that $\g i=j$. Let $F\subset I$ be a finite set such that $\{i,j\}\subseteq F$. Using the first relation of \eqref{401} for $i$ and $j$ in combination with \eqref{402} we get  \begin{equation*}u L(\Sg_j)^{t_j}u^*=\theta(L(H_j))=\theta(u_\g) u L(\Sg_i)^{t_i} u^*\theta(u^*_\g)=v_{\delta(\g)} u v_{\delta(\g)}^*L(\beta_{\delta(\g)}(\Sg_i))^{t_i} v_{\delta(\g)} u^*v^*_{\delta(\g)}.
\end{equation*}
In particular this relation implies that $L(\Sg_j)\prec_M L(\beta_{\delta(\g)}(\Sg_i)) \text{ and }  L(\beta_{\delta(\g)}(\Sg_i))\prec_M L(\Sg_j)$. Since $\Sg_j,\beta_{\delta(\g)}(\Sg_i)$ are normal subgroups of $\La$ these intertwinings combined with \cite[Lemma 2.2]{CI17} imply that $\Sg_j$ is commensurable with $\beta_{\delta(\g)}(\Sg_i)$; in other words \begin{equation}\label{403}[\Sg_j: \Sg_j\cap \beta_{\delta(\g)}(\Sg_i) ]<\infty\text{ and } [\beta_{\delta(\g)}(\Sg_i): \Sg_j\cap \beta_{\delta(\g)}(\Sg_i)]<\infty.
\end{equation}  
Since $\beta_{\delta(\g)}(\Sg_i)\leqslant \oplus_{i\in I} \Sg_i \oplus A$ using the second relation in \eqref{403} there exists a finite subset $j\in J\subset I$ so that $\beta_{\delta(\g)}(\Sg_i)\leqslant \Sg^J \oplus A$. Thus we have the following normal subgroups $\Sg_j\cap \beta_{\delta(\g)}(\Sg_i)\lhd \beta_{\delta(\g)}(\Sg_i)\lhd \Sg^J\oplus A$. Taking the quotient we get a finite normal subgroup\begin{equation*}\beta_{\delta(\g)}(\Sg_i)/\Sg_j\cap \beta_{\delta(\g)}(\Sg_i) \lhd (\Sg^J\oplus A)/\Sg_j\cap \beta_{\delta(\g)}(\Sg_i)= \Sg_j/\Sg_j\cap \beta_{\delta(\g)}(\Sg_i)\oplus \Sg_{J\setminus\{j\}}\oplus A.
\end{equation*}
Hence $\beta_{\delta(\g)}(\Sg_i)/\Sg_j\cap \beta_{\delta(\g)}(\Sg_i) \lhd vZ(\Sg_j/\Sg_j\cap \beta_{\delta(\g)}(\Sg_i)\oplus \Sg_{J\setminus\{j\}}\oplus A) $. However, since  $\Sg_{J\setminus\{j\}}\oplus A)$ is icc by \eqref{403} we have  $vZ(\Sg_j/\Sg_j\cap \beta_{\delta(\g)}(\Sg_i)\oplus \Sg_{J\setminus\{j\}}\oplus A)=\Sg_j/\Sg_j\cap \beta_{\delta(\g)}(\Sg_i)$. Altogether these relations show that $\beta_{\delta(\g)}(\Sg_i)/\Sg_j\cap \beta_{\delta(\g)}(\Sg_i) \lhd \Sg_j/\Sg_j\cap \beta_{\delta(\g)}(\Sg_i)$ and hence $\beta_{\delta(\g)}(\Sg_i)\leqslant \Sg_j$. Similarly one can show that $\beta_{\delta(\g)}(\Sg_i)\geqslant \Sg_j$ and hence $\beta_{\delta(\g)}(\Sg_i)= \Sg_j$. This shows that $\Sg_i \cong \Sg_0$ for all $i$. Moreover there is an action $\Psi\ca I$ which induces a generalized Bernoulli action $\Psi\ca^b \oplus_{i\in I} \Sg_i$. Also since the action of $\Psi\ca^\beta \Sg$ leaves the subgroup $\oplus_I\Sg_i$ invariant then $\Psi$ will also leave invariant $C_\Sg(\oplus_I \Sg_i)=A$. Hence there exists an action $\Psi\ca^\alpha A$ such that $\La=(\Sg_0^{(I)}\oplus A)\rtimes_{b\oplus \alpha} \Psi$. The remaining part of the statement follows directly from the above considerations.\end{proof}

\noindent \emph{Proof of Corollary \ref{maincor2}}.  This follows proceeding in the same manner as in the proof of Corollary \ref{maincor1} and using Theorem \ref{wprig2}. $\hfill\square$

\begin{Remarks}\label{possibleex} When considering generalized Bernoulli actions it is clear the conditions presented in the statements of Theorems \ref{semidirprodrig2}, \ref{wprig2} are satisfied when all the stabilizers of action $\G\ca I$ are finite. 

On the other hand, if one wants to tackle the infinite amenable stabilizers situation, producing examples seems far more challenging. In this direction we would like to present a possible approach for this which was suggested to us by Professor Denis Osin during the AIM workshop ``Classification of group von Neumann algebras''. Consider $\Sigma_0$ an icc finitely generated amenable group. By \cite[Theorem 1.2]{AMO06} there exists an icc supragroup $\Sg_0<\G_0$ that has property (T) and is hyperbolic relatively to $\Sg_0$. Hence by \cite{Oz06} it follows that $\G_0$ is biexact. Let $\G=\G_0\times \G_0$ and consider the diagonal subgroup $\Sg=diag(\Sg_0)<\G$. Also let $\G\ca I= \G/\Sg$ be the action by left multiplication on the right cosets $\G/\Sg$. Since $\Sg_0<\G_0$ is icc and almost malnormal it follows that the one-sided quasinormalizer satisfies $QN^{(1)}_\G(\Sg)=\Sg$. In turn this is equivalent with condition c) in Theorem \ref{strongrigsemidirect1}. Finally one can check that condition b) in Theorems \ref{semidirprodrig2}, \ref{wprig2} is equivalent with the property that the group $\Sg$ has finite height in $\G$ (or it is almost malnormal). This is equivalent to the following property: there exists $k\in \mathbb N$ such that any subset $F<\Sg_0$ with $|F|\geq k$ has finie centralizer $C_{\G_0}(F)$. While f.g. groups like this exist in general (e.g.\ monster groups) it is unclear if one can construct amenable examples. 

In any case a possible positive answer to this last group theoretic question would lead to a class of generalized wreath products constructions with non-amenable core that are recognizable from the von Neumann algebraic setting. Indeed, Theorem \ref{semidirprodrig1} together with the argument from the proof of Claim \ref{contcoreslot} in Theorem \ref{strongrigsemidirect1} give the following 
\end{Remarks}

\begin{cor}\label{wprodrig3}Let $H_0, \G$  be icc, property (T) groups. Also assume that $\G=\G_1\times \G_2$, where $\G_i$ are nonamenable biexact groups for all $i=1,2$. Let $\Gamma \curvearrowright I$ be an action on a countable set $I$ satisfying the following conditions:
\begin{enumerate}
\item [a)] The stabilizer ${\rm Stab}_\G(i)$ is amenable for each $i \in I$;
\item [b)] There is $k\in \mathbb N$ such that for each $J\subseteq I$ satisfying  $|J|\geq k$ we have $|{\rm Stab}_\G(J)|<\infty$.
\item [c)] The orbit ${\rm Stab}_\G(i)\cdot j$ is infinite for all $i\neq j$. 
\end{enumerate}
Denote by $G=H_0 \wr_I \G$ the corresponding generalized wreath product. Let $\La$ be any torsion free group and let $\theta: L(G)\ra L(\La)$ be a $\ast$-isomorphism. Then $\La$ admits a wreath product decomposition $\La = \Sg_0 \wr_I \Psi$ satisfying all the properties enumerated in a)-c). In addition there exist a group isomorphism $\rho:\G \ra \Psi$, a character $\eta:\G\ra \mathbb T$, a $\ast$-isomorphism $\theta_0: L(H_0)\ra L(\Sg_0)$ and a unitary $v\in L(\La)$ such that for every $x\in L(H_0^{(I)})$ and $\g\in \G$ we have \begin{equation*}
\theta(x u_\g) =\eta(\g)v^* \theta^{\bar\otimes I}_0(x) v_{\delta(\g)}v.\end{equation*} 
Here $\{u_\g \,|\, \g\in \G\}$ and $\{v_\la \,|\, \la\in \Psi\}$ are the canonical group unitaries of $L(\G)$ and $L(\Psi)$, respectively. 

\end{cor}


\begin{thebibliography}{999999!}

\bibitem[AMO06]{AMO06} G. Arzhantseva, A. Minasyan, D. Osin, \textit{The SQ-universality and residual properties of relatively hyperbolic groups}, J. of Algebra, {\bf 315} (2007), 165-177.

\bibitem[BO06]{BO06} I. Belegradek, D. Osin, \textit{Rips construction and Kazhdan property (T)}, Groups, Geom., Dynam. {\bf 2} (2008), 1--12.

\bibitem[B14]{B14}M. Berbec, \textit{$W^*$-superrigidity for wreath products with groups having positive first $\ell^2$-Betti number}, Internat. J. Math. {\bf 26} (2015), 1550003, 27 pp.

 \bibitem[BV13]{BV13}M. Berbec, S. Vaes, \textit{$W^*$-superrigidity for group von Neumann algebras of left-right wreath products}, Proc. Lond. Math. Soc. (3) {\bf 108} (2014), 1116--1152.










\bibitem[BKKO14]{BKKO14} E. Breuillard, M. Kalantar, M. Kennedy, N. Ozawa \textit{$C*$-simplicity and the unique trace property for discrete groups}, Publ. Math. Inst. Hautes Études Sci. {\bf 126} (2017), 35--71.

\bibitem[BO08]{BO08} N. P. Brown and N. Ozawa, \textit{$\mathrm{C}^\ast$-algebras and finite-dimensional approximations}, Graduate Studies in Mathematics, vol. 88, AMS, Providence, RI.

\bibitem[Bo12]{Bo12}R, Boutonnet, \textit{$W^*$-superrigidity of mixing Gaussian actions of rigid groups}, Adv. Math., \textbf{244} (2013), 69-9-0. 
\bibitem[BC14]{BC14} R, Boutonnet and A. Carderi, \textit{Maximal amenable von Neumann subalgebras arising from maximal amenable subgroups}, to appear in Geom.  Funct. Anal., ArXiv:1411.4093.
\bibitem[BHR12]{BHR12}R. Boutonnet, C. Houdayer, and S. Raum, \textit{Amalgamated free product type III factors with at most one Cartan subalgebra}.  Compos. Math. \textbf{150} (2014), 143--174.

\bibitem[CdSS16]{CdSS16} I. Chifan, R. de Santiago, and T. Sinclair, \textit{ $W^*$-rigidity for the von Neumann algebras of products of hyperbolic groups}, Geom. Funct. Anal. {\bf 26} (2016), 136-159.




 



\bibitem[CdSS17]{CdSS17} I. Chifan, R. de Santiago, and W. Sucpikarnon, \textit{Tensor product decompositions of II$_1$ fastors arising from extensions of amalgamated free product groups}, arXiv:1710.05019. 

\bibitem[CI08]{CI08}I. Chifan and A. Ioana, \textit{Ergodic subequivalence relations induced by a Bernoulli action}, Geom. Funct. Anal. \textbf{20} (2010), 53--67.
\bibitem[CH08]{CH08} I. Chifan and C. Houdayer, \textit{Bass-Serre rigidity results in von Neumann algebras}, Duke Math. J. \textbf{ 153} (2010), 23--54.
\bibitem [CK14]{CS14}I. Chifan and Y. Kida, \textit{$OE$ and $W^*$ superrigidity results for actions by surface braid groups}, Preprint  February 2014, ArXiv:1502.02391.


\bibitem[CKP14]{CKP14} I. Chifan, Y. Kida, and S. Pant, \textit{Structural results for the von Neumann algebras associated with surface braid groups}, preprint, Int. Math. Res. Not. 2016 (2016), 4807--4848.

\bibitem[CI17]{CI17} I. Chifan, A. Ioana, \textit{Amalgamated free product rigidity for group von Neumann algebras}, to appear in Adv. Math.
(arXiv:1705.07350).

\bibitem[CS11]{CS11} I. Chifan and T. Sinclair, \textit{On the structural theory of $\textrm{II}_1$factors of negatively curved groups}, Ann. Sci. \'{E}c. Norm. Sup. {\bf 46} (2013), 1--33.

\bibitem[CSU11]{CSU11} I. Chifan, T. Sinclair, and B. Udrea, \textit{On the structural theory of $\textrm{II}_1$ factors of negatively curved groups, II. Actions by product groups}, Adv. Math. {\bf 245} (2013), 208--236.

\bibitem [CP10]{CP10} I. Chifan and J. Peterson, \textit{Some unique group measure space decomposition results},   Duke Math. J. \textbf{162} (2013), 1923--1966.
\bibitem[CIK13]{CIK13}I. Chifan, A. Ioana, and Y. Kida, \textit{$W^*$-superrigidity for arbitrary actions of central quotients of braid groups}, Math. Ann. \textbf {361} (2015), 563--582.

\bibitem[CSU13]{CSU13} I. Chifan, T. Sinclair and B. Udrea, \textit{Inner amenability for groups and central sequences in factors}, Ergodic Theory Dynam. Systems \textbf{36} (2016), 1106--1029.

\bibitem[Ch79]{Ch79} E. Christensen, \textit{Subalgebras of a finite algebra}, Math. Ann. {\bf 243} (1979) 17-29. 

\bibitem[Co76]{Co76} A. Connes, {\it Classification of injective factors}, Ann. Math. {\bf 104} (1976) 73-115.

\bibitem[CH89]{CH89} M. Cowling and U. Haagerup, {\it Completely bounded multipliers of the Fourier algebra of a simple Lie group of real rank one}, Invent. Math. {\bf 96} (1989) 507-549.


\bibitem[DHI16]{DHI16} D.~Drimbe, D.~Hoff, A.~Ioana, \textit{Prime II$_1$ factors arising from irreducible lattices in products of rank one simple Lie groups}, arXiv:1611.02209.
















\bibitem[JGS10]{JGS10} J. Fang, S. Gao, and R. Smith, \textit{The Relative Weak Asymptotic Homomorphism Property for Inclusions of Finite von Neumann Algebras}, Intern. J. Math. {\bf 22} (2011) 991--1011.

\bibitem[FV10]{FV10} P. Fima and S. Vaes, \textit{HNN extensions and unique group measure space decomposition of II$_1$ factors}, Trans. Amer. Math. Soc. \textbf{364} (2012), 2601--2617.

 \bibitem[G96]{G96} L. Ge, \textit{On maximal injective subalgebras of factors}, Adv. Math. {\bf 118} (1996), 34--70. 

\bibitem[HPV10]{HPV10}C. Houdayer, S. Popa, and S. Vaes, \textit{A class of groups for which every action is $W^*$-superrigid}, Groups Geom. Dyn. \textbf {7} (2013), 577--590.
\bibitem[HV12]{HV12}C. Houdayer and S. Vaes, \textit{Type III factors with unique Cartan decomposition}, J. Math. Pures Appl. \textbf{100} (2013), 564--590.

\bibitem[HU15]{HU15} C. Houdayer, Y. Ueda, \textit{Rigidity of free product von Neumann algebras}, to appear in Compos. Math. 


\bibitem[Io10]{Io10} A. Ioana, \textit{$W^*$-superrigidity for Bernoulli actions of property (T) groups}, J. Amer. Math. Soc. \textbf{24} (2011), 1175--1226.

\bibitem[Io06]{Io06} A. Ioana, \textit{Rigidity results for wreath product II$_1$ factors}, J. Funct. Anal. {\bf 252} (2007), 763--791.

\bibitem[IPP05]{IPP05} A. Ioana, J. Peterson, and S. Popa, \textit{Amalgamated free products of weakly rigid factors and calculation of their symmetry groups}. Acta Math. {\bf 200} (2008), 85--153. 

\bibitem[Io12a]{Io12a} A. Ioana, \textit{Cartan subalgebras of amalgamated free product II$_1$ factors}, With an appendix by Ioana and Stefaan Vaes. Ann. Sci. Ec. Norm. Super. (4) {\bf 48} (2015), 71--130.
 

\bibitem[Io11]{Io11} A. Ioana, \textit{Uniqueness of the group measure space decomposition for Popa's $\mathcal H\mathcal T$ factors}, Geom. Funct. Anal.  \textbf{22} (2012), 699--732. 

 \bibitem[Io12]{Io12}A. Ioana, \textit{Classification and rigidity for von Neumann algebras}, European Congress of Mathematics, 601-625, Eur. Math. Soc., Zurich, 2013, 46--02. 


\bibitem[Io17]{Io17icm} A. Ioana, \textit{Rigidity for von Neumann algebras}, submitted to Proceedings of the ICM 2018.

\bibitem[IPV10]{IPV10} A. Ioana, S. Popa, and S. Vaes, \textit{A Class of superrigid group von Neumann algebras}, Ann. of Math. (2) {\bf 178} (2013), 231--286.

\bibitem[Is12]{Is12}Y. Isono, \textit{Examples of factors which have no Cartan subalgebras}, to appear in Trans. Amer. Math. Soc., ArXiv1209.1728.
\bibitem[Is14]{Is14}Y. Isono, \textit{Some prime factorization results for free quantum group factors}, to appear in J. Reine Angew. Math., ArXiv:1401.6923.

\bibitem[Is16]{Is16} Y. Isono, \textit{On fundamental groups of tensor product II$_1$ factors}, Preprint 2016, arXiv:1608.06426.








\bibitem[Jo81]{Jo81} V.F.R. Jones, \textit{Index for subfactors}, Invent. Math. {\bf 72} (1983), 1--25.

















\bibitem[KV15]{KV15} A. S. Krogager, S. Vaes, \textit{A class of II$_1$ factors with exactly two group measure space decompositions}, J. Math. Pures Appl. (9) {\bf 108} (2017),  88--110.



\bibitem[Ma79]{Ma79} G.A. Margulis, {\it Finiteness of quotients of discrete groups}, Func. Anal. Appl. {\bf 13} (1979), 178-187.

\bibitem[MvN43]{MvN43} F. J. Murray and  J. von Neumann, {\it On rings of operators}, IV, Ann. of Math. {\bf 44} (1943), 716-808.








\bibitem[Oz03]{Oz03} N. Ozawa, \textit{Solid von Neumann algebras}, Acta Math., {\bf 192} (2004), 111--117. 

\bibitem[Oz05]{Oz05} N. Ozawa, \textit{ A Kurosh type theorem for type II$_1$ factors}, Int. Math. Res. Not., Volume 2006, Article ID97560 (21 pages). 

\bibitem[Oz06]{Oz06} N. Ozawa, {\it Boundary amenability of relatively hyperbolic groups}, Topology Appl. {\bf 53} (2006) 2624-2630.
 
\bibitem[OP03]{OP03} N. Ozawa and S. Popa, \textit{Some prime factorization results for type II$_1$ factors}, Invent. Math., {\bf 156} (2004), 223--234. 

\bibitem [OP07]{OP07} N. Ozawa and S. Popa, \textit{On a class of II$_1$ factors with at most one Cartan subalgebra}, Ann. of Math. \textbf{172} (2010), 713-749.

\bibitem[OP08]{OP08}N. Ozawa and S. Popa, \textit{On a class of II$_1$ factors with at most one Cartan subalgebra. II}, Amer. J. Math. \textbf {132} (2010), 841--866.




\bibitem[Pe06]{Pe06} J. Peterson, \textit{$L^2$-rigidity in von Neumann algebras}, Invent. Math. \textbf{175} (2009), no. 2, 417--433.
\bibitem[Pe09]{Pe09} J. Peterson, \textit{Examples of group actions which are virtually $W*$-superrigid}, Preprint February 2009, ArXiv:1002.1745.





\bibitem[Po94]{Po94} S. Popa, \textit{Classification of subfactors and their endomorphisms}. CBMS Regional Conference Series in Mathematics, 86. Published for the Conference Board of the Mathematical Sciences, Washington, DC; by the American Mathematical Society, Providence, RI, 1995. x+110 pp. 

\bibitem[Po99]{Po99} S. Popa, \textit{Some properties of the symmetric enveloping algebra of a factor, with applications to amenability and property (T)}, Doc. Math. {\bf 4} (1999), 665--744.

\bibitem[Po01]{Po01} S. Popa, \textit{On a class of type II$_1$ factors with Betti numbers invariants}, Ann. of Math. \textbf{163} (2006), 809-899.

\bibitem[Po03]{Po03} S. Popa, \textit{Strong Rigidity of II$_1$ Factors Arising from Malleable Actions of $w$-Rigid Groups I}, Invent. Math. \textbf{165}  (2006), 369--408.
 \bibitem[Po04]{Po04} S. Popa, \textit{Strong Rigidity of II$_1$ Factors Arising from Malleable Actions of $w$-Rigid Groups II}, Invent. Math. \textbf{165} (2006), 409--453.

\bibitem[Po08]{Po08} S. Popa, \textit{On the superrigidity of malleable action with spectal gap}, J. Amer. Math. Soc. \textbf {21} (2008), 981--1000.

\bibitem[Po06]{Po06}S. Popa, \textit{On Ozawa's property for free group factors}, Int. Math. Res. Not. Vol. 2007 : article ID rnm036, 10 pages.

\bibitem[Po05]{Po05} S. Popa, \textit{Cocycle and orbit equivalence superrigidity for malleable actions of $w$-rigid groups}, Invent. Math. {\bf 170} (2007), 243--295.

\bibitem[Po06]{Po06} S. Popa, \textit{Deformation and rigidity for group actions and von Neumann algebras}, International Congress of Mathematicians. Vol. I, 445--477, Eur. Math. Soc., Z\"urich, 2007.




\bibitem[PV09]{PV09} S. Popa and S. Vaes, \textit{Group measure space decomposition of II$_1$ factors and $W^*$-superrigidity}, Invent. Math. {\bf 182} (2010), 371--417.

\bibitem[PV11]{PV11}S. Popa and S. Vaes, \textit{Unique Cartan decomposition for $\textrm{II}_1$ factors arising from arbitrary actions of free groups}, Acta Math. \textbf{212} (2014), 141--198.

\bibitem[PV12]{PV12} S. Popa and S. Vaes, \textit{Unique Cartan decomposition for $\textrm{II}_1$ factors arising from arbitrary actions of hyperbolic groups}, J. Reine Angew. Math. {\bf 694} (2014), 215--239. 

\bibitem[Si10]{Si10} T. Sinclair, \textit{Strong solidity of group factors from lattices in $SO(n,1)$ and $SU(n,1)$}, J. Funct. Anal. \textbf{260} (2011), no. 11, 3209--3221.
\bibitem[SW12]{SW12} O. Sizemore and A. Winchester, \textit{A unique prime decomposition result for wreath product factors},  Pacific J. Math. \textbf{265} (2013), no. 1, 221--232.






\bibitem[Va07]{Va07} S. Vaes, \textit{Explicit computations of all finite index bimodules for a family of II$_1$ factors}, Ann. Sci. \'{E}c. Norm. Sup. \textbf{41} (2008), 743-788.

\bibitem[V10]{V10} S. Vaes, \textit{Rigidity for von Neumann algebras and their invariants}, Proceedings of the International Congress of Mathematicians (Hyderabad, India, 2010), Vol. III, 1624--1650, Hindustan Book Agency, New Delhi, 2010.

\bibitem[Va10]{Va10} S Vaes, \textit{One-cohomology and the uniqueness of the group measure space decomposition of a II$_1$ factor}, Math. Ann. \textbf{355} (2013), 661-696.

\bibitem[Va13]{Va13} S. Vaes, \textit{Normalizers inside amalgamated free products von Neumann algebras}, Publ. Res. Inst. Math. Sci. \textbf{50} (2014), 695--721. 
\bibitem[VV14]{VV14} S. Vaes and P. Verraedt, \textit{Classification of type III Bernoulli crossed products}, Adv. Math. \textbf{281} (2015), 296--332. 






















 \bibitem[V96]{V96}D-V. Voiculescu,\textit{The analogues of entropy and of Fisher's information measure in free probability theory: the absence of Cartan subalgebras}, Geom. Funct. Anal. \textbf{6} (1996), no. 1,172--199.



\end{thebibliography}
\end{document}